\documentclass{amsart}
\usepackage{amsmath,amssymb,amsthm,mathrsfs,verbatim}
\usepackage{caption}
\usepackage{subcaption}
\usepackage{mathrsfs}
\usepackage[all]{xy}
\usepackage{graphicx}

\newtheorem{Thm}{Theorem}[section]
\newtheorem{Prop}[Thm]{Proposition}

\newtheorem{Lem}[Thm]{Lemma}
\newtheorem{Def}[Thm]{Definition}
\newtheorem{Conj}[Thm]{Conjecture}

\theoremstyle{definition}
\newtheorem{Ex}[Thm]{Example}
\newtheorem{Rem}[Thm]{Remark}

\newcommand{\C}{\mathbb{C}}

\newcommand{\R}{\mathbb{R}}
\newcommand{\Z}{\mathbb{Z}}

\newcommand{\Q}{\mathbb{Q}}
\newcommand{\bS}{\mathbb{S}}
\newcommand{\PP}{\mathbb{P}}
\newcommand{\bP}{\mathbb{P}}

\newcommand{\der}{\mathrm{d}}
\newcommand{\GL}{\mathop{\mathrm{GL}}\nolimits}
\newcommand{\U}{\mathop{\mathrm{U}}\nolimits}

\newcommand{\Spec}{\mathop{\mathrm{Spec}}\nolimits}
\newcommand{\ev}{\mathop{\mathrm{ev}}\nolimits}
\newcommand{\Hol}{\mathop{\mathrm{Hol}}\nolimits}
\newcommand{\pt}{\mathrm{pt}}

\newcommand{\cD}{\mathcal{D}}
\newcommand{\tZ}{\tilde{Z}}

\begin{document}
\title{Geometric transitions and SYZ mirror symmetry} 
\author{Atsushi Kanazawa \ \ \ Siu-Cheong Lau}
\date{}

\begin{abstract}
We prove that generalized conifolds and orbifolded conifolds are mirror symmetric under the SYZ program with quantum corrections.  
Our work mathematically confirms the gauge-theoretic assertion of Aganagic--Karch--L\"ust--Miemiec,  
and also provides a supportive evidence to Morrison's conjecture that geometric transitions are reversed under mirror symmetry. 
\end{abstract}

\maketitle


\section{Introduction}

In \cite{Mor}, Morrison proposed that geometric transitions are reversed under mirror symmetry. 
A geometric transition is a birational contraction followed by a complex smoothing, or in the reverse way, applied to a K\"ahler manifold (see a nice review \cite{Ros} by Rossi) .  
We will denote a geometric transition by $\widehat{X} \dashrightarrow X \rightsquigarrow \widetilde{X}$, 
where $\widehat{X} \dashrightarrow X$ is a birational contraction and $X \rightsquigarrow\widetilde{X}$ is a smoothing. 
The conjecture can be formulated as follows.

\begin{Conj}[Morrison \cite{Mor}]
Let $\widehat{X}$ and $\widetilde{X}$ be Calabi--Yau manifolds, and suppose they are related by a geometric transition $\widehat{X} \dashrightarrow X \rightsquigarrow \widetilde{X}$. 
Suppose $Y_1$ and $Y_2$ are the mirrors of $\widehat{X}$ and $\widetilde{X}$ respectively. 
Then there exists a geometric transition $Y_2 \dashrightarrow Y \rightsquigarrow Y_1$ relating $Y_1$ and $Y_2$.
\end{Conj}

The present paper investigates mirror symmetry for geometric transitions of two specific types of local singularities, namely generalized conifolds and orbifolded conifolds. 
Let us first recall mirror symmetry for a conifold. 
A conifold is an isolated singularity defined by $\{xy-zw=0\} \subset \C^4$. 
It is an important singularity appearing in algebraic geometry and also plays a special role in superstring theory. 
A folklore mirror symmetry for the conifold \cite{Mor,Sze} that the deformed conifold is mirror symmetric to the resolved conifold 
can be refined in the framework of SYZ mirror symmetry as follows. 
\begin{Thm}[Conifold case of Theorem \ref{MainTheorem}]
Let $\widehat{X}:=\mathcal{O}_{\PP^1}(-1)^{\oplus2}\setminus D$ be the resolved conifold with a smooth anti-canonical divisor $D$ removed, and 
$$
\widetilde{X}:=\{(x,y,z,w) \in \C^4 \ | \  xy-zw=1\} \setminus ( \{z = 1\} \cup \{w = 1\})
$$
the deformed conifold with the anti-canonical divisor $\{z = 1\} \cup \{w = 1\}$ removed. 
Then $\widehat{X}$ and $\widetilde{X}$ are SYZ mirror to each other.  
\end{Thm}
Although removing the divisors certainly does not affect the local geometry of the singularity, 
it is important when we discuss, for example, wrapped Fukaya categories and homological mirror symmetry\footnote{We are grateful to Murad Alim for informing us about the importance of this issue.}. 
 
We now focus on two natural generalizations of the conifold: generalized conifolds and orbifolded conifolds. 
For integers $k,l \ge 1$, a generalized conifold is given by 
$$
G_{k,l}^\sharp:=\{(x,y,z,w) \in \C^4\ | \  xy-(1+z)^k(1+w)^l=0\}
$$
and an orbifolded conifold is given by 
$$
O_{k,l}^\sharp:=\{(u_1,v_1,u_2,v_2,z) \in \C^5 \ | \  u_1v_1-(1+z)^k=u_2v_2-(1+z)^l=0\}. 
$$
(We have made a change of coordinates, namely $z \mapsto 1+z$ and $w \mapsto 1+w$, for later convenience.)  They reduce to the conifold when $k=l=1$. 
The punctured generalized conifold is defined as $G_{k,l}:=G_{k,l}^\sharp\setminus D_G$, 
where $D_{G}:=\{z = 0\} \cup \{w=0\}$ is a normal-crossing anti-canonical divisor of $G_{k,l}^\sharp$, 
and the punctured orbifolded conifold as $O_{k,l}:=O_{k,l}^\sharp\setminus D_O$, where $D_{O}:=\{z = 0\}$ is a smooth anti-canonical divisor of $O_{k,l}^\sharp$. 
As is the case of the conifold, their symplectic structures and complex structures are governed by the crepant resolutions and deformations respectively. 
The main theorem of the present paper is the following.  

\begin{Thm}[Theorem \ref{MainTheorem}] \label{MainTheoremIntro}
The punctured generalized conifold $G_{k,l}$ is mirror symmetric to the punctured orbifolded conifold  $O_{k,l}$ in the sense that 
the deformed punctured generalized conifold $\widetilde{G_{k,l}}$ is SYZ mirror symmetric to the resolved punctured orbifolded conifold $\widehat{O_{k,l}}$, 
and the resolved punctured generalized conifold $\widehat{G_{k,l}}$ is SYZ mirror symmetric to the deformed punctured orbifolded conifold $\widetilde{O_{k,l}}$. 
$$
\xymatrix{
\widetilde{G_{k,l}}\ar@{<->}[d]_{SYZ} &\ar@{~>}[l] G_{k,l} \ar@{<->}[d]^{MS} & \ar@{.>}[l]  \widehat{G_{k,l}} \ar@{<->}[d]^{SYZ}  \\
\widehat{O_{k,l}} \ar@{.>}[r] & O_{k,l}\ar@{~>}[r] & \widetilde{O_{k,l}}. 
}
$$
\end{Thm}
According to Theorem \ref{MainTheoremIntro} the mirror duality of the conifold is purely caused by the fact that the conifold can be seen as either a generalized or an orbifolded conifold. 
The mirror duality of $G_{k,l}^\sharp$ and $O_{k,l}^\sharp$ has previously been studied by physicists Aganagic, Karch, Lust and Miemiec in \cite{AKLM}, where they use gauge theory and brane configurations. 
Our work mathematically confirms their gauge-theoretic assertion, and also provides a supportive evidence to Morrison's conjecture that geometric transitions are reversed under mirror symmetry. 

In the present paper, we use the framework introduced by the second author with Chan and Leung \cite{CLL} for defining SYZ mirror pairs. 
Namely, generating functions of open Gromov--Witten invariants of fibers of a Lagrangian fibration were used to construct the complex coordinates of the mirror. 
The essential ingredient is wall-crossing of the generating functions, which was first studied by Auroux \cite{Aur}. 
We can also bypass symplectic geometry and employs the Gross--Siebert program \cite{GS} which uses tropical geometry instead for defining mirror pairs. 
This tropical approach was taken by Castano-Bernard and Matessi \cite{CM} in the study of conifold transitions for compact Calabi--Yau varieties. 
Although our work has some overlap with theirs, the methods and interests are quite different, and we also investigate more general situations. 
In this paper tropical geometry is unnecessary since symplectic geometry can be handled directly.  




One crucial feature of the present work is the involutive property of SYZ mirror symmetry. 
Namely, taking SYZ mirror twice gets back to itself, which we believe is an important point but often overlooked in literatures.
We exhibit this feature by carrying out the SYZ construction for all the four directions in Theorem \ref{MainTheoremIntro}, 
namely from $\widehat{G_{k,l}}$ to $\widetilde{O_{k,l}}$, and from $\widetilde{O_{k,l}}$ back to $\widehat{G_{k,l}}$; 
from $\widetilde{G_{k,l}}$ to $\widehat{O_{k,l}}$, and from $\widehat{O_{k,l}}$ back to $\widetilde{G_{k,l}}$. 
The SYZ construction from $\widehat{G_{k,l}}$ to $\widetilde{O_{k,l}}$ is a bit tricky and we will discuss it in details. 
We will employ the various techniques developed in \cite{Aur, CLL, AAK, Lau}. 

Another interesting feature is the dependence of the choice of a Lagrangian fibration, 
namely a Lagrangian fibration has to be compatible with the choice of an anti-canonical divisor in order to obtain the desired mirror. 
For example, $\widehat{G_{k,l}}$ admits two different Lagrangian fibrations: the Gross fibration and a `doubled' Gross fibration.  
The former is not compatible with the anti-canonical divisor $D_G$, and does not produce the orbifold conifold as its SYZ mirror. 
Choosing appropriate Lagrangian fibrations is a key step in our work. 

Lastly, Theorem \ref{MainTheoremIntro} not only unveils a connection between geometric transitions and SYZ mirror symmetry, 
but also yields many interesting problems and conjectures that naturally extend what is known for the conifolds. 
In fact, based on the local models studied in this paper, the second author recently confirmed Morrison's conjecture for a class of geometric transitions of the Schoen's Calabi--Yau threefolds \cite{Lau2}.  


\subsection*{Structure of Paper}
Section \ref{GenOrbConi} introduces generalized conifolds and orbifolded conifolds and basic properties thereof. 
Section \ref{SYZ} begins with a review on the Lagrangian torus fibrations and the SYZ program.  
Then we prove the main theorem (Theorem \ref{MainTheoremIntro}) by carrying out the SYZ constructions. 
Section \ref{Discussion} discusses global geometric transitions and provides a few examples.

\subsection*{Acknowledgement}
The authors are grateful to Murad Alim, Kwokwai Chan, Yu-Wei Fan, Conan Leung, Shing-Tung Yau for useful discussions and encouragement. 
A.K. was supported by the Harvard CMSA and S.-C.L. was supported by Harvard University when the present work was done. 


\section{Generalized and orbifolded conifolds} \label{GenOrbConi}
In this section, we introduce two natural generalizations of the conifold, namely generalized conifolds and orbifolded conifolds. 
These two singularities possess interesting geometries and were studied by physicists in the context of gauge theory, for instance in \cite{KKV,AKLM,Mie}. 


\subsection{Generalized conifolds $G_{k,l}^\sharp$}
A toric Calabi--Yau threefold can be described by a lattice polytope $\Delta \subset \R^2$ whose vertices lie in the lattice $\Z^2 \subset \R^2$.  Its fan is produced by taking the cone over $\Delta \times \{1\} \subset \R^3$.
A crepant resolution of a toric Calabi--Yau threefold corresponds to a subdivision of $\Delta$ into standard triangles\footnote{
A standard triangle in $\R^2$ is isomorphic to the convex hull of $(0,0),(1,0),(0,1)$ under the $\Z^2 \rtimes \GL(2,\Z)$-transformation.}, which gives a refinement of the fan. 
For instance, the total space of the canonical bundle $K_S$ of a smooth toric surface $S$ is a toric Calabi--Yau threefold.  
In this situation, the surface $S$ is the toric variety $\PP_\Delta$ whose fan polytope is $\Delta$. 

The condition that a toric Calabi--Yau threefold contains no compact $4$-cycles is equivalent to the condition that the polytope $\Delta$ contains no interior lattice points. 
The lattice polygons without interior lattice point are classified, up to the action of $\GL(2,\Z)$, into two types: 
\begin{enumerate}
\item triangle with vertices $(0,0), (2,0), (0,2)$,
\item trapezoid $\Delta_{k,l}$ with vertices $(0,0), (0,1), (k,0), (l,1)$ for $k \ge l \ge 0$ with $(k,l)\ne (0,0)$ (Figure \ref{fig:GeneralizedConifold}(a)).
\begin{figure}[htbp]
 \begin{center} 
  \includegraphics[width=100mm]{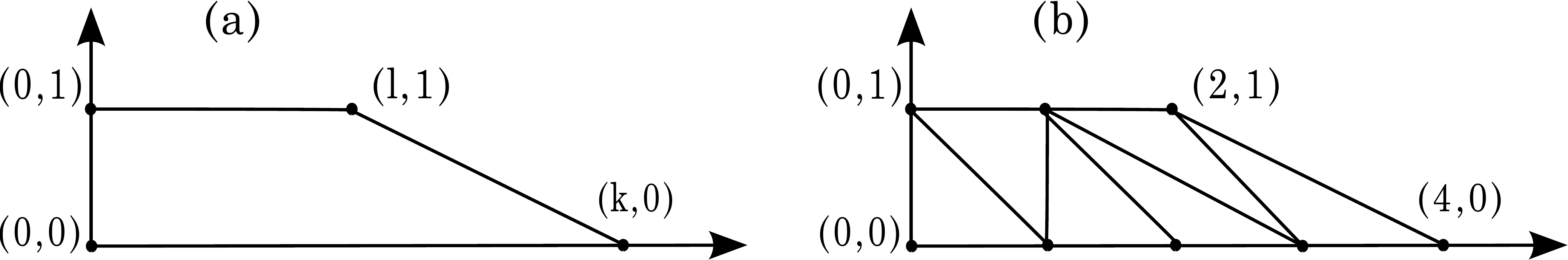}
 \end{center}
 \caption{(a) Trapezoid  $\Delta_{k,l}$, \ (b) Crepant resolution $\widehat{G_{4,2}^\sharp}$}
\label{fig:GeneralizedConifold}
\end{figure}
\end{enumerate}
The former is the quotient of $\C^3$ by the subgroup $(\Z_2)^2\subset \mathrm{SL}(3,\C)$ generated by the two elements $\mathrm{diag}(-1,-1,1)$ and $\mathrm{diag}(1,-1,-1)$. 
In this paper, we are interested in the latter, which corresponds to the generalized conifold $G_{k,l}$ for $k \ge l \ge 1$. 
We do not consider the case $l=0$, where the toric singularity essentially comes from the $A_k$-singularity in 2-dimensions\footnote{
Mirror symmetry of this class of singularities is discussed in \cite[Section 5]{Sze}.}.  
The dual cone of the cone over the trapezoid $\Delta_{k,l}$ is spanned by the vectors
\begin{equation} \label{eq:nu}
\nu_1:=(1,0,0), \ \nu_2:=(0,-1,1), \ \nu_3:=(-1,l-k,k), \ \nu_4:=(0,1,0)
\end{equation}
with relation $\nu_1-k\nu_2+\nu_3-l\nu_4=0$. 
In equation the generalized conifold $G_{k,l}^\sharp$ is given by 
$$
G_{k,l}^\sharp:=\{xy-z^{k}w^{l}=0\} \subset \C^4. 
$$
The coordinates $x,y,z,w$ correspond to the dual lattice points $\nu_1,\nu_3,\nu_2,\nu_4$ respectively.
For $(k,l)\ne (1,1)$, the generalized conifold $G_{k,l}^\sharp$ is a quotient of the conifold (which is given by $(k,l)=(1,1)$) and has a $1$-dimensional singular locus. 
A punctured generalized conifold is $G_{k,l}:=G_{k,l}^\sharp\setminus D_G$, where $D_G=\{z=1\} \cup \{w=1\}$ is an anti-canonical divisor of $G_{k,l}^\sharp$.  
A crepant resolution $\widehat{G_{k,l}}$ of $G_{k,l}$ is called a resolved generalized conifold. 
We observe that $\widehat{G_{k,l}}=\widehat{G_{k,l}^\sharp}\setminus D_{\widehat{G}}$, where $D_{\widehat{G}}$ is an anti-canonical divisor of $\widehat{G_{k,l}^\sharp}$, 
and it uniquely corresponds to a maximal triangulation of the trapezoid $\Delta_{k,l}$ (Figure \ref{fig:GeneralizedConifold}(b)). 
The resolved generalized conifold $\widehat{G_{k,l}}$ is endowed with a natural symplectic structure as an open subset of a smooth toric variety $\widehat{G_{k,l}^\sharp}$.  

\begin{Prop} \label{triangulation}
There are $\binom{k+l}{k}$ distinct crepant resolutions of $G_{k,l}$ (or equivalently $G_{k,l}^\sharp$).  
\end{Prop}
\begin{proof}
There is a bijection between the crepant resolutions of $G_{k,l}$ and the maximal triangulations of $\Delta_{k,l}$. 
The assertion easily follows by induction with the relation $\binom{k+l+1}{k+1}=\binom{k+l}{k}+\binom{k+l}{k+1}$. 
\end{proof}

We may also smooth out the punctured generalized conifold $G_{k,l}$ by deforming the equation. 
The deformed generalized conifold $\widetilde{G_{k,l}}$ is defined as
$$
\widetilde{G_{k,l}}:=\Big\{ (x,y,z,w) \in \C^2 \times (\C \setminus\{1\})^2 \ \big| \ xy- \sum_{i=0}^{k}\sum_{j=0}^{l}a_{i,j}z^iw^j=0\Big\}
$$
for generic $a_{i,j} \in \C$. 
The symplectic structure of $\widetilde{G_{k,l}}$ is given by the restriction of the standard symplectic structure on $\C^2 \times (\C \setminus\{1\})^2$. 
We observe that the complex deformation space has dimension $(k+1)(l+1)-3$ because three of the parameters can be eliminated by rescaling $z,w$ and rescaling the whole equation.  
On the other hand, the K\"ahler deformation space has dimension $(k+1)+(l+1)-3$, the number of linearly dependent lattice vectors in the polytope. 
It is the number of the exceptional $\PP^1$s' and a K\"ahler form is parametrized by the area of these. 


\subsection{Orbifolded conifolds $O_{k,l}^\sharp$}
Let $X^\sharp$ be the conifold $\{xy-zw=0\} \subset \C^4$. 
For $k \ge l \ge 1$, the orbifolded conifold $O_{k,l}^\sharp$ is the quotient of the conifold $X^\sharp$ by the abelian group $\Z_k \times \Z_l$, where $\Z_k$ and $\Z_l$ respectively act by
$$
(x,y,z,w) \mapsto (\zeta_kx,\zeta_k^{-1}y,z,w), \textrm{ and } (x,y,z,w) \mapsto (x,y,\zeta_lz,\zeta_l^{-1}w)
$$
where $\zeta_k$, $\zeta_l$ are primitive $k$-th and $l$-th roots of unity respectively (assume $\mathrm{gcd}(k,l)=1$ for simplicity \cite{AKLM}). 
Alternatively the orbifolded conifold $O_{k,l}^\sharp$ is realized as a hypersurface in $\C^5$: 
$$
O_{k,l}^\sharp=\{u_1v_1=z^k, \ u_2v_2=z^l \} \subset \C^5. 
$$ 
The orbifolded conifold $O_{k,l}^\sharp$ is an example of a toric Calabi--Yau threefold 
and the corresponding polytope is given by the rectangle $\Box_{k,l}$ with the vertices $(0,0), (k,0), (0,l), (k,l)$ (Figure \ref{fig:OrbifoldedConifold}(a)). 
\begin{figure}[htbp]
 \begin{center} 
  \includegraphics[width=90mm]{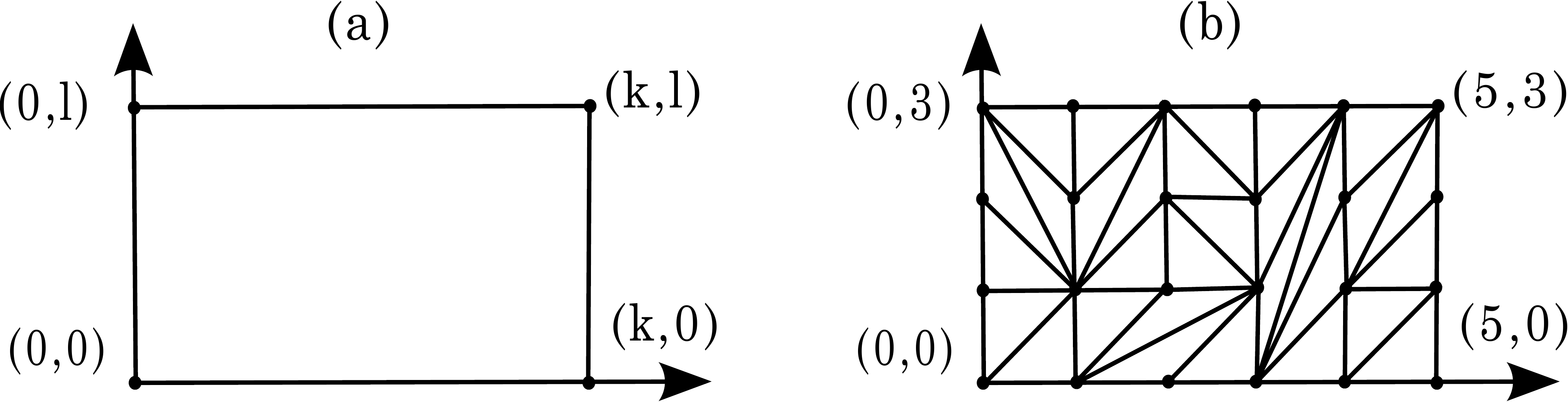}
 \end{center}
 \caption{(a) Rectangle $\Box_{k,l}$, \ (b) Crepant resolution $\widehat{O_{5,3}^\sharp}$}
\label{fig:OrbifoldedConifold}
\end{figure}

The dual cone of the cone over the rectangle $\Box_{k,l}$ is spanned by the following vectors 
$$
v_1:=(1,0,0), \ v_2:=(0,-1,l), \ v_3:=(-1,0,k), \ v_4:=(0,1,0) 
$$
with relation $lv_1-kv_2+lv_3-kv_4=0$. \\

A punctured orbifolded conifold is $O_{k,l}:=O_{k,l}^\sharp\setminus D_{O}$, where $D_{O}=\{z=1\}$ is a smooth anti-canonical divisor of $O_{k,l}^\sharp$.  
Then a resolved orbifolded conifold $\widehat{O_{k,l}}$ is defined to be a crepant resolution of $O_{k,l}$. 
As before, $\widehat{O_{k,l}}=\widehat{O_{k,l}^\sharp}\setminus D_{\widehat{O}}$, 
where $D_{\widehat{O}}$ is a smooth anti-canonical divisor of the toric crepant resolution $\widehat{O_{k,l}^\sharp}$,
and it corresponds to a maximal triangulation of the trapezoid $\Box_{k,l}$ (Figure \ref{fig:OrbifoldedConifold}(b)). 
It has a canonical symplectic structure as an open subset of a smooth toric variety $\widehat{O_{k,l}^\sharp}$.   
In contrast to Proposition \ref{triangulation}, it is a famous open problem to find the number of the crepant resolutions of the orbifolded conifold $O_{k,l}^\sharp$ \cite{KZ}. 
The punctured orbifolded conifold $O_{k,l}$ can also be smoothed out by deforming the defining equations. 
Thus the deformed orbifolded conifold $\widetilde{O_{k,l}}$ is given by 
$$
\widetilde{O_{k,l}}:=\Big\{(u_1,v_1,u_2,v_2,z) \in \C^4 \times (\C \setminus \{1\}) \ \big| \ u_1v_1=\sum_{i=0}^k a_iz^i, \ u_2v_2=\sum_{j=0}^l b_jz^j\Big\}
$$
for generic coefficients $a_i,b_j \in \C$. 
The symplectic structure of $\widetilde{O_{k,l}}$ is the restriction of the standard symplectic structure on $\C^4 \times (\C \setminus \{1\})$.  
The complex deformation space of $\widetilde{O_{k,l}}$ has dimension $(k+1)+(l+1)-3$, while the K\"ahler deformation space has dimension $(k+1)(l+1)-3$. 
Therefore the naive dimension counting is compatible with our claim that these two classes of singularities are mirror symmetric. 
We will formulate this mirror duality in a rigorous manner by using SYZ mirror symmetry in the next section.


\section{SYZ mirror construction} \label{SYZ}

The Strominger--Yau--Zaslow (SYZ) conjecture \cite{SYZ} provides a foundational geometric understanding of mirror symmetry. 
It asserts that, for a mirror pair of Calabi--Yau manifolds $X$ and $X^\vee$, 
there exist Lagrangian torus fibrations $\pi:X\to B$ and $\pi^\vee:X^\vee\to B$ which are fiberwise-dual to each other. 
In particular, it suggests an intrinsic construction of the mirror $X^\vee$ by fiberwise dualizing a Lagrangian torus fibration on $X$. 
This is motivated by \textit{T-duality} studied by string theorists. 

The SYZ program has been carried out successfully in the semi-flat case \cite{Leu} in which the discriminant locus of the fibrations is empty. 
When singular fibers are present, quantum corrections by open Gromov--Witten invariants of the fibers are necessary, and they exhibit wall-crossing phenomenon.  Wall-crossing of open Gromov--Witten invariants was first studied by Auroux \cite{Aur}.  Later on \cite{CLL} gave an SYZ construction of mirrors with quantum corrections, which will be used in this paper.
In algebro-geometric context, the Gross--Siebert program \cite{GS} gives a reformulation of the SYZ program using tropical geometry, which provides powerful techniques to compute wall-crossing and scattering order-by-order.  In this paper we will use the symplectic rather than the tropical approach.

We will first give a quick review of the setting of \cite{CLL} for SYZ with quantum corrections in Section \ref{Preliminary}.
We say that $X$ is \emph{SYZ mirror symmetric} to $Y$ if $Y$ is produced from $X$ as a SYZ mirror manifold by this SYZ mirror construction.  
The later parts of this section prove the following main theorem.
 
\begin{Thm} \label{MainTheorem}
The punctured generalized conifold $G_{k,l}$ is mirror symmetric to the punctured orbifolded conifold $O_{k,l}$ in the sense that 
the deformed generalized conifold $\widetilde{G_{k,l}}$ is SYZ mirror symmetric to the resolved orbifolded conifold $\widehat{O_{k,l}}$, 
and the resolved generalized conifold $\widehat{G_{k,l}}$ is SYZ mirror symmetric to the deformed orbifolded conifold $\widetilde{O_{k,l}}$:   
$$
\xymatrix{
\widetilde{G_{k,l}}\ar@{<->}[d]_{SYZ} &\ar@{~>}[l] G_{k,l} \ar@{<->}[d]^{MS} & \ar@{.>}[l]  \widehat{G_{k,l}} \ar@{<->}[d]^{SYZ}  \\
\widehat{O_{k,l}} \ar@{.>}[r] & O_{k,l}\ar@{~>}[r] & \widetilde{O_{k,l}}.
}
$$
\end{Thm}
This mathematically confirms the gauge-theoretic assertion of the string theorists Aganagic--Karch--L\"ust--Miemiec \cite{AKLM} 
and also provides a supportive evidence to Morrison's conjecture \cite{Mor} from the view point of SYZ mirror symmetry. 





\subsection{SYZ construction with quantum corrections} \label{Preliminary}
In this subsection we review the SYZ construction with quantum corrections given in \cite{CLL}.  We add a clarification that we only use transversal disc classes (Definition \ref{def:transversal}) in the definition of the mirror space.

Let $\pi: X \to B$ be a proper Lagrangian torus fibration of a K\"ahler manifold $(X,\omega)$  
such that the base $B$ is a compact manifold with corners, and the preimage of each codimension-one facet of $B$ is a smooth irreducible divisor denoted as $D_i$ for $1 \le i \le m$.
We assume that the regular Lagrangian fibers of $\pi$ are special with respect to a nowhere-vanishing meromorphic volume form $\Omega$ on $X$ whose pole divisor is the boundary divisor $D:=\sum_{i=1}^mD_i$ (and hence $D$ is an anti-canonical divisor).  We denote by $B_0 \subset B$ the complement of the discriminant locus of $\pi$, and we assume that $B_0$ is connected\footnote{When the discriminant locus has codimension-two, $B_0$ is automatically connected.  Although the Lagrangian fibrations of $\widetilde{G_{k,l}}$ we study have codimension-one discriminant loci, $B_0$ is still connected.}.  We always denote by $F_b$ a fiber of $\pi$ at $b \in B_0$.

\begin{Lem}[Maslov index of disc classes {\cite[Lemma 3.1]{Aur}}] \label{MaslovIndex} 
For a disc class $\beta \in \pi_2(X,F_b)$ where $b \in B_0$, the Maslov index of $\beta$ is $\mu(\beta)=2D\cdot \beta$. 
\end{Lem}


\begin{Def}[Wall \cite{CLL}] \label{def:wall}
The wall $H$ of a Lagrangian fibrartion $\pi:X\rightarrow B$ is the set of point $b \in B_0$ 
such that $F_b:=\pi^{-1}(b)$ bounds non-constant holomorphic disks with Maslov index $0$.  
\end{Def}

The complement of $H \subset B_0$ consists of several connected components, which we call chambers. 
Over different chambers the Lagrangian fibers behave differently in a Floer-theoretic sense. 
Away from the wall $H$, the one-point open Gromov--Witten invariants are well-defined using the machinery of Fukaya--Oh--Ohta--Ono \cite{FOOO}. 

\begin{Def}[Open Gromov--Witten invariants {\cite{FOOO}}] \label{def:oGW}
For $b \in B_0 \setminus H$ and $\beta \in \pi_2(X,F_b)$, 
let  $\mathfrak{M}_1(\beta)$ be the moduli space of stable discs with one boundary marked point representing $\beta$, and $[\mathfrak{M}_1(\beta)]^{\mathrm{virt}}$ be the virtual fundamental class of $\mathfrak{M}_1(\beta)$.
The open Gromov--Witten invariant associated to $\beta$ is $n_\beta:=\int_{[\mathfrak{M}_1(\beta)]^{\mathrm{virt}}}\ev^*[\pt]$, 
where $\ev:\mathfrak{M}_1(\beta) \rightarrow F_b$ is the evaluation map at the boundary marked point and $[\pt]$ is the Poincar\'e dual of the point class of $F_b$.
\end{Def}
We will restrict to disc classes which are transversal to the boundary divisor $D$ when we construct the mirror space (while for the mirror superpotential we need to consider all disc classes).

\begin{Def}[Transversal disc class] \label{def:transversal}
A disc class $\beta \in \pi_2(X,F_b)$ for $b \in B_0$ is said to be transversal to the boundary divisor $D$, which is denoted as $\beta \pitchfork D$, if all stable discs in $\mathfrak{M}_1(\beta)$ intersect transversely with the boundary divisor $D$.
\end{Def}

Due to dimension reason, the open Gromov--Witten invariant $n_\beta$ is non-zero only when the Maslov index $\mu(\beta)=2$. 
When $\beta$ is transversal to $D$ or when $X$ is semi-Fano, namely $c_1(\alpha) = D \cdot \alpha \geq 0$ for all holomorphic sphere classes $\alpha$, the number $n_\beta$ is invariant under small deformation of complex structure and under Lagrangian isotopy in which all Lagrangian submanifolds in the isotopy do not intersect $D$ nor bound non-constant holomorphic disc with Maslov index $\mu(\beta)<2$. 


The paper \cite{CLL} proposed a procedure which realizes the SYZ program based on symplectic geometry as follows:

\begin{enumerate}
\item
Construct the semi-flat mirror $X^\vee_0$ of $X_0:=\pi^{-1}(B_0)$ as the space of pairs $(b,\nabla)$ 
where $b \in B_0$ and $\nabla$ is a flat $\U(1)$-connection on the trivial complex line bundle over $F_b$ up to gauge. 
There is a natural map $\pi^\vee:X^\vee_0\rightarrow B_0$ given by forgetting the second coordinate. 
The semi-flat mirror  $X^\vee_0$ has a canonical complex structure \cite{Leu} 
and the functions $\mathrm{e}^{-\int_{\beta}\omega}\Hol_{\nabla}(\partial \beta)$ on $X^\vee_0$ for disc classes $\beta \in \pi_2(X,F_b)$ are called semi-flat complex coordinates. 
Here $\mathrm{Hol}_{\nabla} (\partial \beta)$ denotes the holonomy of the flat $\U(1)$-connection $\nabla$ along the path $\partial \beta \in \pi_1(F_b)$. 

\item Define the generating functions of open Gromov--Witten invariants for $1 \le i \le m$
\begin{equation}
Z_i(b,\nabla) := \sum_{\substack{\beta \in \pi_2(X,F_b) \\ \beta \cdot D_i = 1, \beta\pitchfork D}} n_\beta \mathrm{e}^{-\int_{\beta}\omega}\Hol_{\nabla}(\partial \beta), 
\label{eq:gen}
\end{equation}
for $(b, \nabla) \in (\pi^\vee)^{-1}(B_0\setminus H)$, which serve as quantum corrected complex coordinates. 
The function $Z_i$ can be written in terms of the semi-flat complex coordinates, and hence they generate a subring $\C[Z_1, \ldots, Z_m]$ in the function ring\footnote{In general we need to use the Novikov ring instead of $\C$ since $Z_i$ could be a formal Laurent series.  In the cases that we study later, $Z_i$ are Laurent polynomials whose coefficients are convergent, and hence the Novikov ring is not necessary.} over $(\pi^\vee)^{-1}(B_0\setminus H)$.

\item
Define the SYZ mirror of $X$ with respect to the Lagrangian torus fibration $\pi$ to be the pair $(X^\vee,W)$ where 
$X^\vee:=\Spec \left(\C[Z_1,\dots, Z_m] \right)$ and 
$$W = \sum_{\substack{\beta \in \pi_2(X,F_b)}} n_\beta \mathrm{e}^{-\int_{\beta}\omega}\Hol_{\nabla}(\partial \beta). $$
\end{enumerate}

Moreover, $X^\vee$ is defined as the SYZ mirror of a non-compact Calabi--Yau manifold $Y$ if it is obtained from the above construction for a compactification of $Y$. 
It is expected that different compactifications would result in the same SYZ mirror. 
In this paper we fix one compactification as an initial data for the SYZ construction.





In the following sections we will apply the above recipe to the generalized conifolds and orbifolded conifolds.  
We will carry out in detail the SYZ construction from $\widehat{G_{k,l}}$ to $\widetilde{O_{k,l}}$ which is the most interesting case (Section \ref{sec:GtoO}), in which we construct a doubled version of the Gross fibration \cite{Gol,Gro} and compute the open Gromov--Witten invariants.  The other cases, namely the SYZ constructions from $\widetilde{O_{k,l}}$ to $\widehat{G_{k,l}}$, from $\widehat{O_{k,l}}$ to $\widetilde{G_{k,l}}$, and from $\widehat{G_{k,l}}$ to $\widetilde{O_{k,l}}$, are essentially obtained by applying the techniques developed in \cite{Lau,CLL,AAK}, and so we will be brief. 
In fact $G_{k,l}^\sharp$ and $O_{k,l}^\sharp$ are useful testing grounds for the SYZ program and we shall illustrate how these various important ideas fit together by examining them. 



\subsection{SYZ from $\widehat{G_{k,l}}$ to $\widetilde{O_{k,l}}$} \label{sec:GtoO}
We first construct the SYZ mirror of the resolved generalized conifold $\widehat{G_{k,l}}$. 
While the resolved generalized conifold $\widehat{G_{k,l}^\sharp}$ is a toric Calabi--Yau threefold,  
we will {\it not} use the Gross fibration \cite{Gol,Gro} because it is not compatible with the chosen anti-canonical divisor $D_G$ and hence do not produce the resolved orbifolded conifold $\widetilde{O_{k,l}}$ as the mirror.   
We will instead use a {\it doubled} version of the Gross fibration explained below.

The fan of $\widehat{G_{k,l}^\sharp}$ is given by the cone over a triangulation depicted in Figure \ref{fig:GeneralizedConifold}(b).  
We label the divisors corresponding to the rays generated by $(i,0,1)$ to be $D_i$ for $0 \le i \le k$, and those corresponding to the rays generated by $(j,1,1)$ to be $D_{k+1+j}$ for $0 \le j \le l$.  
Each divisor $D_i$ corresponds to a basic disc class $\beta_i \in \pi_2(X,L)$ where $L$ denotes a moment-map fiber \cite{CLL}. 

Let us first compactify the resolved generalized conifold $\widehat{G_{k,l}}$ as follows. 
We add the rays generated by $(0,0,-1)$, $(0,-1,-1)$, $(1,0,0)$ and $(-1,0,0)$, and the corresponding cones, to the fan of $\widehat{G_{k,l}^\sharp}$. 
Let us denote the resulting toric variety by $\widehat{G_{k,l}}^*$ (and we fix a toric K\"ahler form on it). 
Let $\cD_{z=\infty}$, $\cD_{w=\infty}$,$\cD_{\xi=0}$, and $\cD_{\xi=\infty}$ be the corresponding additional toric prime divisors and $\beta_{z=\infty}$, $\beta_{w=\infty}$, $\beta_{\xi=0}$ and $\beta_{\xi=0}$ be the additional basic disc classes respectively.  

Note that $\widehat{G_{k,l}}^*$ is in general not semi-Fano since there could be holomorphic spheres with the Chern class $c_1 < 0$ supported in the newly added divisors (or in other words the fan polytope of $\widehat{G_{k,l}}^*$ may contain a interior lattice point).  However since we only need to consider transversal disc classes (Definition \ref{def:transversal}) in the definition of $X^\vee$, these holomorphic spheres do not enter into our constructions.

We now construct a special Lagrangian fibration and apply the SYZ construction on $\widehat{G_{k,l}}^*$. 
Consider the Hamiltonian $T^1$-action on $\widehat{G_{k,l}}^*$ corresponding to the vector $(1,0,0)$ in the vector space which supports the fan.  Denote by $\pi_{T^1}:\widehat{G_{k,l}}^* \to \R$ the moment map associated to this Hamiltonian $T^1$-action, whose image is a closed interval $I$.  Let $\theta$ be the angular coordinate corresponding to the Hamiltonian $T^1$-action. 
Recall that $x,y,z,w$ are toric functions corresponding to the lattice points $\nu_1, \nu_3,\nu_2,\nu_4$ (in the vector space which supports the moment polytope) defined by Equation \eqref{eq:nu} respectively. 
Note that $z=0$ on the toric divisors $D_0, \ldots, D_k$, while $w=0$ on the toric divisors $D_{k+1},\ldots,D_{k+l+1}$.  Moreover the pole divisors of $z$ and $w$ are $\cD_{z=\infty}$ and $\cD_{w=\infty}$ respectively.

The toric K\"ahler form on $\widehat{G_{k,l}}^*$ can be written as
$$
\omega := \der \pi_{T^1} \wedge \der \theta + \frac{\sqrt{-1}}{c_1(1+|z|^2)^2} \der z \wedge \der\bar{z} + \frac{\sqrt{-1}}{c_2(1+|w|^2)^2} \der w \wedge \der \bar{w}
$$
for some $c_1,c_2 \in \R_{>0}$.
We define a $T^3$-fibration $\pi:\widehat{G_{k,l}}^*\rightarrow B:=[-\infty,\infty]^2 \times I$ by 
$$
\pi(x,y,z,w) = (b_1,b_2,b_3) = (\log|z-1|, \log|w-1|,\pi_{T^1}(x,y,z,w)).
$$
We also define a nowhere-vanishing meromorphic volume form by 
$$
\Omega : = \der \log x \wedge \der \log(z-1) \wedge \der \log(w-1) .
$$  
The pole divisor $D$ of $\Omega$ is given by the union 
$$
D = \cD_{z=1} \cup \cD_{w=1} \cup \cD_{z=\infty} \cup \cD_{w=\infty} \cup \cD_{\xi=0} \cup \cD_{\xi=\infty}
$$
whose image under $\pi$ is the boundary of $B$ ($\cD_{z=1}$ and $\cD_{w=1}$ denotes the divisors $\{z=1\}$ and $\{w=1\}$ respectively).  
Using the method of symplectic reductions \cite{Gol}, we obtain the following. 
\begin{Prop}
The $T^3$-fibration $\pi$ defined above is a special Lagrangian fibration with respect to $\omega$ and $\Omega$.
\end{Prop}

\begin{proof}
Consider the symplectic quotient of the Hamiltonian $T^1$-action: $\widetilde{M} := \pi_{T^1}^{-1}(\{b_3\}) / T^1$ for certain $b_3 \in \R$. 
Since the toric coordinates $z$ and $w$ are invariant under the $T^1$-action, they descend to the quotient $\widetilde{M}$.  
This gives an identification of $\widetilde{M}$ with $\PP^1 \times \PP^1$. 
The induced symplectic form on the quotient $\widetilde{M}$ is given by
\begin{align*}
\widetilde{\omega} &= \frac{\sqrt{-1}}{c_1(1+|z|^2)^2} \der z \wedge \der \bar{z} + \frac{\sqrt{-1}}{c_2(1+|w|^2)^2} \der w \wedge \der \bar{w}\\
&= \frac{\sqrt{-1} |z-1|^2}{c_1(1+|z|^2)^2} \der \log (z-1) \wedge \der \overline{\log(z-1)} + \frac{\sqrt{-1} |w-1|^2}{c_2(1+|w|^2)^2} \der \log (w-1) \wedge \der \overline{\log(w-1)}.
\end{align*}
The induced holomorphic volume form $\widetilde{\Omega}$, which is the contraction of $\Omega$ by the vector field induced from the $T^1$-action, 
equals to
$$
\widetilde{\Omega} = \der \log (z-1)\wedge \der \log (w-1).
$$ 
It is clear that that $\widetilde{\omega}$ and $\mathrm{Re}(\widetilde{\Omega})$ restricted on each fiber of the fibration $(|z-1|,|w-1|)$ are both zero.   
Hence the fibers of the map $(|z-1|,|w-1|)$ are special Lagrangian in $\widetilde{M}$.  
By \cite[Lemma 2]{Gol}, we therefore conclude that the fibers of $\pi$ are special Lagrangian. 
\end{proof}

We may think of this fibration as the combination of a conic bundle (Figure \ref{fig:ConicAmoeba2}) 
\begin{figure}[htbp]
 \begin{center} 
  \includegraphics[width=100mm]{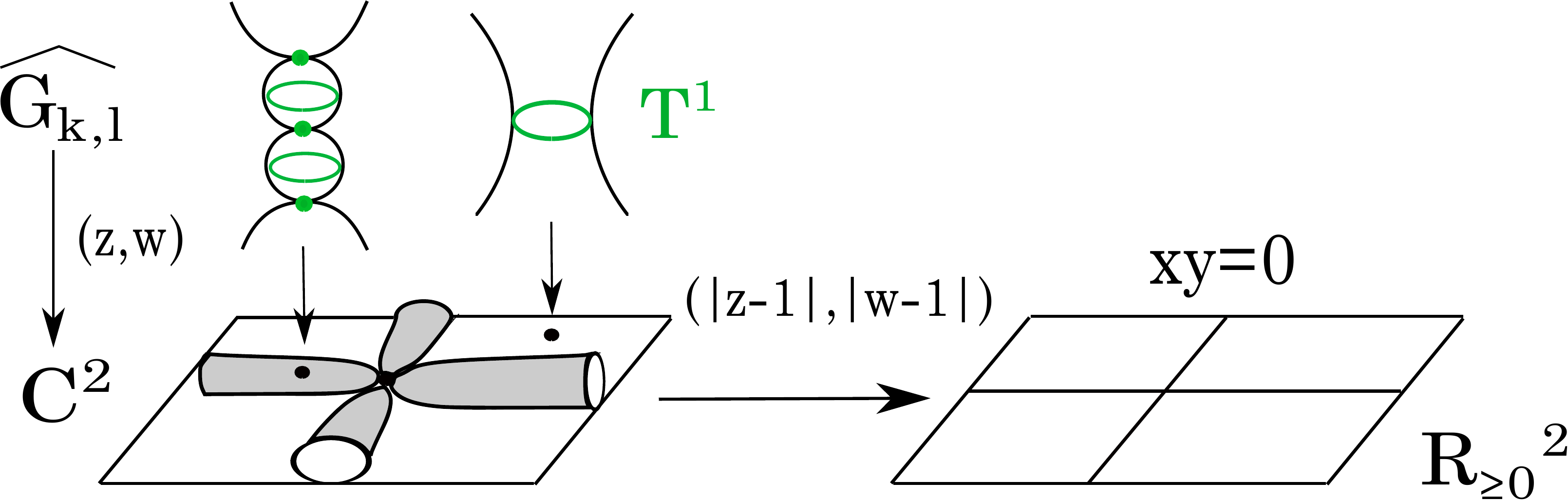}
 \end{center}
 \caption{Conic fibration after resolution}
\label{fig:ConicAmoeba2}
\end{figure}
and the moment map $\pi_{T^1}$ associated to the lift of $(x,y)$-coordinates (Figure \ref{fig:MomentMapGenConi}(a)). 
\begin{figure}[htbp]
 \begin{center} 
  \includegraphics[width=100mm]{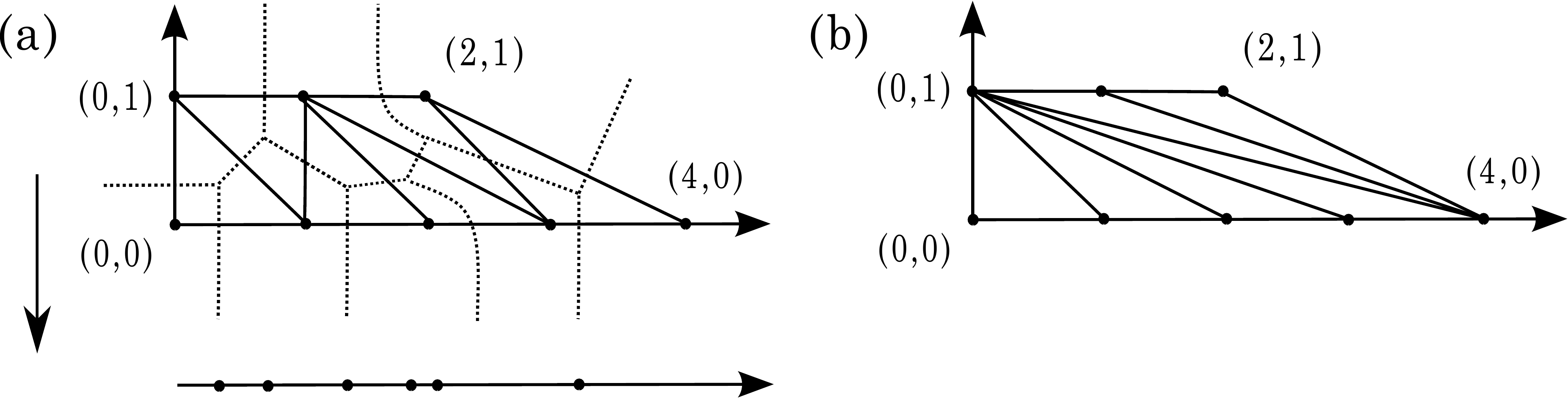}
 \end{center}
 \caption{(a) Moment map $\pi_{T^1}$, (b) Crepant resolution}
\label{fig:MomentMapGenConi}
\end{figure}
The latter measures the volumes of the exceptional curves $\PP^1$ of the crepant resolution $\widehat{G_{k,l}} \rightarrow G_{k,l}$.  

\begin{Prop}
The discriminant locus of the fibration $\pi$ is the union of the boundary $\partial B$ together with the lines 
$\{b_1=0, b_3=s_i\}_{i=1}^k \cup \{b_2=0, b_3=t_j\}_{j=1}^l \subset B$ for $s_i,t_j \in \R$ with $\mathrm{Crit}(\pi_{T^1})=\{s_1,\dots,s_k,t_1,\dots,t_l\}$.  
\end{Prop}

\begin{proof}
The first and second coordinates of $\pi$ are $b_1=\log|z-1|$ and $b_2=\log|w-1|$ respectively, 
which degenerates over the boundaries $b_1 = \log|z-1| = \pm\infty$ or $b_2 = \log|w-1| = \pm\infty$. 
The third coordinate $\pi_{T^1}$ degenerates at those codimension-two toric strata whose corresponding 2-dimensional cones in the fan contain the vector $(1,0,0)$. 
These cones are either $[i-1,i] \times \{0\} \times \R$ for $1 \le i\le k$ or $[j-1,j] \times \{1\} \times \R$ for $1 \le j\le l$.  
The corresponding images under $\pi_{T^1}$ are isolated points $s_1,\dots,s_k$ or $t_1,\dots,t_l$ respectively.  Moreover $z=0$ on a toric strata corresponding to a cone $[i-1,i] \times \{0\} \times \R$, while $w=0$ on a toric strata corresponding to a cone $[j-1,j] \times \{1\} \times \R$.  Hence $b_1=1$ or $b_2=1$ respectively, and the discriminant locus is as stated above.
\end{proof}

\begin{Prop}
The wall $H$ of the fibration $\pi$ is given by the union of two vertical planes given by $b_1=0$ and $b_2=0$.  
\end{Prop}
\begin{proof}
Suppose a fiber $F_r$ bounds a non-constant holomorphic disc $u$ of Maslov index $0$.  By the Maslov index formula in Lemma \ref{MaslovIndex}, the disc does not intersect the boundary divisors $\{z = 0\}$ nor $\{w=0\}$.  Thus the functions $(z-1) \circ u$ and $(w-1) \circ u$ can only be constants.   If both the numbers $z \circ u$ and $w \circ u$ are non-zero, the fiber of $(z,w)$ is just a cylinder, and a fiber of $b_3$ defines a non-contractible circle in this cylinder, which topologically does not bound any non-trivial disc.  Thus either $z = 0$ or $w = 0$ on the disc, which implies that  $b_1=\log|z-1|=0$ or $b_2=\log|w-1|=0$.  In these cases $F_r$ intersects a toric divisor and bounds holomorphic discs in the toric divisor.
\end{proof}

Figure \ref{fig:Wall5}(a) illustrates the wall stated in the above proposition.
\begin{figure}[htbp]
 \begin{center} 
  \includegraphics[width=100mm]{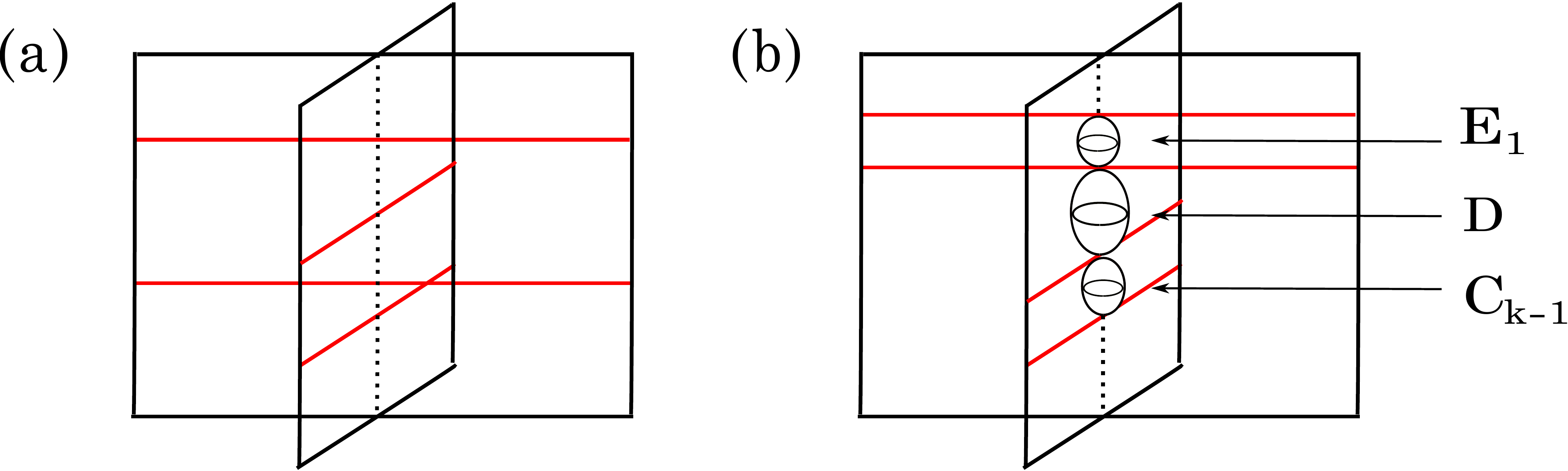}
 \end{center}
 \caption{(a) Walls, (b) Holomorphic spheres}
\label{fig:Wall5}
\end{figure}

From now on we fix the unique crepant resolution of $G_{k,l}$ such that $s_1 < \ldots < s_k < t_1 \ldots < t_l$ holds (Figure \ref{fig:MomentMapGenConi}(b)).  For other crepant resolution the construction is similar (while the SYZ mirrors have different coefficients, namely the mirror maps are different).
Such a choice is just for simplifying the notations and is not really necessary (see also Remark \ref{DifferentCrepantResolutions}). 

We then fix the basis 
\begin{equation} \label{eq:C&E}
\{C_i\}_{i=1}^{k-1} \cup \{C_0\} \cup \{E_i\}_{i=1}^{l-1}
\end{equation}
of $H_2(\widehat{G_{k,l}})$ (Figure \ref{fig:Wall5}(b)), where
$C_i$ for $1 \le i \le k-1$ is the holomorphic sphere class represented by the toric $1$-stratum corresponding to the $2$-cone  by $\{(0,1),(i,0)\}$; $C_0$ corresponds to the $2$-cone generated by $\{(0,1),(k,0)\}$; $E_i$ corresponds to the $2$-cone generated by $\{(i,1),(k,0)\}$.  The image of a holomorphic sphere in $C_i$ under the fibration map lies in $\{0\} \times \R \times [s_i,s_{i+1}]$; 
the image of $C_0$ lies in $\{0\} \times \{0\} \times [s_k,t_1]$, and  
the image of $E_i$ lies in $\R \times \{0\} \times [t_i,t_{i+1}]$. \\

Fix the contractible open subset
$$
U := B_0 \setminus \left\{(b_1,b_2,b_3) \ | \ b_1=0 \text{ or } b_2=0, b_3 \in [s_1,+\infty) \right\} \subset B_0
$$
over which the Lagrangian fibration $\pi$ trivializes.  
For $b = (b_1,b_2,b_3)$ with $b_1>0$ and $b_2 > 0$, we use the Lagrangian isotopy
\begin{equation} \label{eq:isotopy}
L^t = \{ \log|z-t| = b_1, \log|w-t| = b_2, \pi_T(x,y,z,w) = b_3\}
\end{equation}
for $t \in [0,1]$ to link a moment-map fiber (when $t=0$) with a Lagrangian torus fiber $F_b$ of $\pi$ (when $t=1$).  Then for a general base point $b' \in U$, we can link the fibers $F_b$ and $F_{b'}$ by a Lagrangian isotopy induced by a path joining $b$ and $b'$ in the contractible set $U$ (and the isotopy is independent of choice of the path).
Through the isotopy disc classes bounded by a moment-map fiber $L$ can be identified with those bounded by $F_b$, 
that is, $\pi_2(\widehat{G_{k,l}}^*,F_b) \cong \pi_2(\widehat{G_{k,l}}^*,L)$. 
Note that this identification depends on choice of trivialization, and henceforth we fix such a choice. 
The two vertical walls $\{b_1=0\}$ and $\{b_2=0\}$ divides the base $B =  [-\infty,\infty]^2 \times I$ into four chambers.  
Lagrangian torus fibers over different chambers have different open Gromov--Witten invariants.

\begin{Thm} \label{thm:G-oGW_wc}
Denote by $L$ a moment-map fiber of $\widehat{G_{k,l}}^*$ and by $F_b$ a Lagrangian torus fiber of $\pi$ at $b\in B_0$. 
Let $\beta \in \pi_2(\widehat{G_{k,l}}^*,F_b)$ with $\beta \pitchfork D$. 
\begin{enumerate}
\item Over the chamber $C_{++}:=\{b_1>0, b_2>0\}$, we have $n_\beta^{F_b} = n_\beta^L$.

\item Over the chamber $C_{+-}:=\{b_1>0, b_2<0\}$, we have $n_\beta^{F_b} = 0$ unless $\beta = \beta_{k+1}$, $\beta_{\xi=0}$, $\beta_{\xi=\infty}$, $\beta_{z=\infty}$, $\beta_{w=\infty} + (\beta_j - \beta_{k+1}) + \alpha$ for $k+1\le j \le k+l+1$ and $\alpha \in H_2^{c_1=0}$ being a class of rational curves which intersect the open toric orbit of the toric divisor $D_j$, or $\beta = \beta_i + \alpha$ for $0 \le i \le k$ and $\alpha \in H_2^{c_1=0}$ being a class of rational curves which intersect the open toric orbit of the toric divisor $D_i$.  Moreover
$$n^{F_b}_{\beta_{k+1}} = n^{F_b}_{\beta_{\xi=0}} = n^{F_b}_{\beta_{\xi=\infty}} = n^{F_b}_{\beta_{z=\infty}} = 1, \ \ \ 
n^{F_b}_{\beta_{w=\infty} + (\beta_j - \beta_{k+1}) + \alpha} = n^L_{\beta_j + \alpha}$$
for $k+1\le j \le k+l+1$, and
$$n_{\beta_i + \alpha}^{F_b} = n_{\beta_i + \alpha}^L$$
for $0 \le i \le k$.

\item Over the chamber $C_{-+}:=\{b_1<0, b_2>0\}$, we have $n_\beta^{F_b} = 0$ unless $\beta = \beta_0$, $\beta_{\xi=0}$, $\beta_{\xi=\infty}$, $\beta_{w=\infty}$, $\beta_{z=\infty} + (\beta_i - \beta_{0}) + \alpha$ for $0 \le i \le k$ and $\alpha \in H_2^{c_1=0}$ being a class of rational curves which intersect the open toric orbit of the toric divisor $D_i$, 
or $\beta = \beta_j + \alpha$ for $k+1\le j \le k+l+1$ and $\alpha$ being a class of rational curves which intersect the open toric orbit of the toric divisor $D_j$.  Moreover 
$$n^{F_b}_{\beta_0} = n^{F_b}_{\beta_{\xi=0}} = n^{F_b}_{\beta_{\xi=\infty}} = n^{F_b}_{\beta_{w=\infty}} = 1, \ \ \
n^{F_b}_{\beta_{z=\infty} + (\beta_i - \beta_0) + \alpha} = n^L_{\beta_i + \alpha}$$
for $0 \le i \le k$, and
$$n_{\beta_j + \alpha}^{F_b} = n_{\beta_j + \alpha}^L$$
for $k+1 \le j \le k+l+1$.

\item Over the chamber $C_{--}:=\{b_1<0, b_2<0\}$, we have $n_\beta^{F_b} = 0$ unless $\beta = \beta_0$, $\beta_{k+1}$, $\beta_{\xi=0}$, $\beta_{\xi=\infty}$, $\beta_{z=\infty} + (\beta_i - \beta_{0}) + \alpha$ for $0 \le i \le k$ and $\alpha \in H_2^{c_1=0}$ being a class of rational curves which intersect the open toric orbit of the toric divisor $D_i$, or $\beta_{w=\infty} + (\beta_j - \beta_{k+1}) + \alpha$ for $k+1 \le j \le k+l+1$ and $\alpha \in H_2^{c_1=0}$ being a class of rational curves which intersect the open toric orbit of the toric divisor $D_j$.  Moreover
$$n^{F_b}_{\beta_0} = n^{F_b}_{\beta_{k+1}} = n^{F_b}_{\beta_{\xi=0}} = n^{F_b}_{\beta_{\xi=\infty}} = 1, \ \ \ 
n^{F_b}_{\beta_{z=\infty} + (\beta_i - \beta_0) + \alpha} = n^L_{\beta_i + \alpha}$$
for $0 \le i \le k$, and
$$n^{F_b}_{\beta_{w=\infty} + (\beta_j - \beta_{k+1}) + \alpha} = n^L_{\beta_j + \alpha}$$
for $k+1 \le j \le k+l+1$.
\end{enumerate}
\end{Thm}

\begin{proof}
The open Gromov--Witten invariants $n_\beta$ is non-zero only when $\beta$ has Maslov index $2$, and so we can focus on $\mu(\beta)=2$ with $\beta \pitchfork D$.

For a fiber $F_{(b_1,b_2,b_3)}$ with $b_1>0$ and $b_2>0$, $F_{(b_1,b_2,b_3)}$ is Lagrangian isotopic to a moment-map fiber $L$ through $L_t$ defined by Equation \eqref{eq:isotopy}.  
Moreover each $L_t$ does not bound any holomorphic disc of Maslov index $0$ because for every $t \in [0,1]$, 
the circles $|z-t| = b_1$ and $|w-t|=b_2$ never pass through $z=0$ and $w=0$ respectively.  
Thus the open Gromov--Witten invariants of $L$ and that of $F_{(b_1,b_2,b_3)}$ are the same. 

Now consider the chamber $C_{+-}$.  First we use the Lagrangian isotopy 
$$
L^{1,t} = \{ \log|z-1| = b_1, \log|w-t| = b_2, \pi_T(x,y,z,w) = b_3\}
$$
for $t \in [1, R]$ which identifies $F_{(b_1,b_2,b_3)}$ with $L^{1,R}$ for $R \gg 0$.  $L^{1,t}$ never bounds any holomorphic disc of Maslov index $0$, since for every $t \in [1,R]$, 
the circles $|z-1| = b_1$ and $|w-t|=b_2$ never pass through $z=0$ and $w=0$ respectively.

Then we take the involution $\iota: \widehat{G_{k,l}}^* \to \widehat{G_{k,l}}^*$ defined as identity on $z, \pi_{T^1}, \theta$ and mapping $w \mapsto \frac{Rw}{w-R}$. 
This involution maps the fiber $L^{1,R}$ to the Lagrangian
$$L' = \{ \log|z-1| = b_1, \log|w-R| = 2(\log R) - b_2, \pi_T(x,y,z,w) = b_3\}$$
which can again be identified with the fiber $F_{(b_1, 2 (\log R) - b_2,b_3)}$ with $b_1>0$ and $2 (\log R) - b_2 > 0$. 
Also $\iota$ tends to the negative identity map as $R$ tends to infinity. 
Hence for $R \gg 0$, the pulled-back complex structure by $\iota$ is a small deformation of the original complex structure, 
and hence the open Gromov--Witten invariants of $F_{(b_1,b_2,b_3)}$ remain invariant. 
Now using Case $1$ the open Gromov--Witten invariants of $F_{(b_1,2 (\log R) - b_2,b_3)}$ can be identified with a moment-map fiber $L$.

By considering the intersection numbers of the disc classes and the divisors, one can check that the disc classes $\beta_i$ for $0 \le i \le k$, $\beta_{k+1}$, $\beta_{\xi=0}$, $\beta_{\xi=\infty}$, $\beta_{z=\infty}$, $(\beta_j - \beta_{k+1})$ for $k+1 \le j \le k+l+1$, and all rational curve classes $\alpha$ are invariant under the involution.  Moreover $\beta_{k+1}$ and $\beta_{w=\infty}$ are switched under the involution.  Putting all together, we obtain a complete relation between open Gromov--Witten invariants of $F_{(b_1,b_2,b_3)}$ and that of $L$, and this gives the formulae in Case $2$.

Cases $3$ and $4$ are similar.  For Case $3$ we use the Lagrangian isotopy 
$$
L^{t,1} = \{ \log|z-t| = b_1, \log|w-1| = b_2, \pi_T(x,y,z,w) = b_3\}
$$
for $t \in [1, R]$ and the involution defined as identity on $w, \pi_{T^1}, \theta$ and mapping $z \mapsto \frac{Rz}{z-R}$.  
For Case 4 we use the Lagrangian isotopy 
$$
L^{t} = \{ \log|z-t| = b_1, \log|w-t| = b_2, \pi_T(x,y,z,w) = b_3\}
$$
for $t \in [1,R]$ and the involution defined as identity on $\pi_{T^1}, \theta$ and mapping $z \mapsto \frac{Rz}{z-R}$, $w \mapsto \frac{Rw}{w-R}$. 

\end{proof}

We can compute the the open Gromov--Witten invariants of the moment-map fiber, using the open mirror theorem \cite[Theorem 1.4 (1)]{CCLT}.  
The result is essentially the same as the one in \cite[Theorem 4.2]{LLW} for the minimal resolution of $A_n$-singularities as follows. 

\begin{Thm} \label{thm:G-oGW_toric}
Let $L$ be a regular moment-map fiber of $\widehat{G_{k,l}}$, and consider a disc class $\beta \in \pi_2(X,L)$.  
The open Gromov--Witten invariant $n_\beta^L$ equals to $1$ if $\beta = \beta_p + \alpha$ for $0 \le p \le k+l+1$, where $\alpha$ is a class of rational curves which takes the form
$$ \alpha = \left\{\begin{array}{ll} 
\sum_{i=1}^{k-1} s_i C_i & \textrm{ when } 1 \leq p \leq k-1; \\
\sum_{i=1}^{l-1} s_i E_i & \textrm{ when } k+2 \leq p \leq k+l;\\
0 & \textrm{ when } p = 0, k, k+1 \textrm { or } k+l+1,
\end{array}
\right.$$
and $\{s_i\}_{i=1}^{m-1}$ (where $m$ equals to $k$ in the first case and $l$ in the second case) is an admissible sequence with center $p$ in the first case, which means that
\begin{enumerate}
\item $s_i \geq 0$ for all $i$ and $s_{1}, s_{m-1} \leq 1$; 
\item $s_i \leq s_{i+1} \leq s_i + 1$ when $i < p$, and $s_i \ge s_{i+1} \ge s_i-1$ when $i \ge p$, 
\end{enumerate}
and with center $p-k-1$ in the second case.
For any other $\beta$, $n_\beta^L = 0$.
\end{Thm}

\begin{proof}
We will prove the assertion by using the open mirror theorem.  
Recall the curve classes $C_i$, $C_0$ and $E_j$ introduced in Equation \eqref{eq:C&E} for $1 \le i\le k-1$ and $1 \le j\le l-1$.  
$C_i$ and $E_i$ are $(-2,0)$-curves, while $C_0$ is a $(-1,-1)$-curve.  The intersection numbers with the toric prime divisors $D_j$ are as follows:
\begin{enumerate}
\item $C_i \cdot D_{i-1} = C_i \cdot D_{i+1} = 1; C_i \cdot D_i = -2$; and $C_i \cdot D_j = 0$ for all $ j \not= i-1,i,i+1$;
\item $C_0 \cdot D_k = C_0 \cdot D_{k+1} = -1; C_0 \cdot D_{k-1} = C_0 \cdot D_{k+2} = 1$; and $C_0 \cdot D_j=0$ for all $j \not= k-1,k,k+1,k+2$;
\item $E_i \cdot D_{k+i-1} = E_i \cdot D_{k+i+1} = 1; E_i \cdot D_{k+i} = -2$; and $E_i \cdot D_j = 0$ for all $j \not= k+i-1,k+i,k+i+1$. 
\end{enumerate}

Let $q^{C_i},q^{C_0},q^{E_i}$ be the corresponding K\"ahler parameters and $\check{q}^{C_i},\check{q}^{C_0},\check{q}^{E_i}$ be the corresponding complex parameters.  
They are related by the mirror map:
$$
q^{C_i} = \check{q}^{C_i} \exp \Big(-\sum_{j=0}^{k+1+l} (C_i \cdot D_j) g_j(\check{q}) \Big) 
=  \check{q}^{C_i} \exp \left(- \left(g_{i-1}(\check{q}) + g_{i+1}(\check{q}) - 2 g_i(\check{q})\right) \right),
$$
$$
q^{C_0} = \check{q}^{C_0} \exp \Big(-\sum_{j=0}^{k+1+l} (C_0 \cdot D_j) g_j(\check{q}) \Big) 
= \check{q}^{C_0} \exp \left(-\left(g_{k-1}(\check{q}) + g_{k+1}(\check{q})\right) \right),
$$
and
$$
q^{E_i} = \check{q}^{E_i} \exp \Big(-\sum_{j=0}^{k+1+l} (E_i \cdot D_j) g_j(\check{q}) \Big) 
= \check{q}^{E_i} \exp \left(-\left(g_{k+i-1}(\check{q}) + g_{k+i+1}(\check{q}) - 2 g_{k+i}(\check{q})\right) \right).
$$
The functions $g_i(\check{q})$ are attached to the toric prime divisor $D_i$ for $0 \le i\le k+l+1$.  
We have $g_0 = g_k = g_{k+1} = g_{k+l+1} = 0$.  
Moreover for $1 \le i\le k-1$, $g_i$ only depends on the variables $\check{q}^{C_r}$ for $1 \le r \le k-1$; for $k+2 \le i \le k+l$, 
the function $g_i$ only depends on the variables $\check{q}^{E_r}$ for $1 \le r \le l-1$. 
Explicitly $g_i$ is written in terms of hypergeometric series:
$$
g_i(\check{q}):=\sum_{\substack{d \cdot D_i < 0\\d \cdot D_r \geq 0 \\ \textrm{for all } r \not= i}}\frac{(-1)^{(D_i\cdot d)}(-(D_i\cdot d)-1)!}{\prod_{p\neq i} (D_p\cdot d)!}\check{q}^d
$$
where for $1 \le i\le k-1$, 
the summation is over $d = \sum_{r=1}^{k-1} {n_j C_j}$ ($n_j \in \Z_{\geq 0}$) with $d \cdot D_i < 0$ and $d \cdot D_p \geq 0$ for all $p \not= i$; 
for $k+2 \le i\le k+l$, the summation is over $d = \sum_{r=1}^{l-1} {n_j E_j}$ ($n_j \in \Z_{\geq 0}$) with $d$ satisfying the same condition.
Then the open mirror theorem \cite[Theorem 1.4 (1)]{CCLT} states that
$$ \sum_{\alpha} n_{\beta_i + \alpha} q^{\alpha} (\check{q}) = \exp g_i(\check{q}). $$
Note that for $1 \le i\le k-1$, the function $g_i$ takes exactly the same expression as that in the toric resolution of $A_{k-1}$-singularity; for $1 \le i\le l-1$, the function $g_{i+k+1}$ takes exactly the same expression as that in the toric resolution of $A_{l-1}$-singularity.  Thus the mirror maps for $q^{C_i}$ and $q^{E_i}$ coincide with that for $A_{k-1}$-resolution and $A_{l-1}$-resolution respectively.  Moreover the above generating function of open Gromov--Witten invariants coincide.  Then result follows from the formula for open Gromov--Witten invariants in $A_{n}$-resolution given in \cite[Theorem 4.2]{LLW}.

\end{proof}

\begin{Rem}
Theorem \ref{thm:G-oGW_toric} can also be proved by comparing the disc moduli for $\widehat{G_{k,l}}$ and resolution of $A_k$- and $A_l$-singularities, 
which involves details of obstruction theory of disc moduli space.  
Here take the more combinatorial approach using open mirror theorems. 
\end{Rem}

\begin{Thm}
The SYZ mirror of the resolved generalized conifold $\widehat{G_{k,l}}$ is given by the deformed orbifolded conifold in $\C^4 \times \C^\times$ defined by the equations ($U_1,U_2,V_1,V_2 \in \C$ and $Z\in\C^\times$)
\begin{align}
U_1V_1&=(1+Z)(1+q_1Z)\dots(1+q_1\dots q_{k-1}Z), \notag \\
U_2V_2&=(1+cZ)(1+q'_1cZ)\dots(1+q_1'\dots q_{l-1}'cZ), \notag
\end{align}
where $q_i=\mathrm{e}^{- \int_{C_i}\omega}$, $q_j'=\mathrm{e}^{- \int_{E_i}\omega}$, and $c=q_1 \ldots q_{k-1} \mathrm{e}^{- \int_{C_0}\omega}$. 
\end{Thm}

\begin{proof}
Let $\tZ_\beta = \mathrm{e}^{-\int_{\beta}\omega}\Hol_{\nabla}(\partial \beta)$ be the semi-flat mirror complex coordinates corresponding to each disc class $\beta \in \pi_2(\widehat{G_{k,l}}^*,F_b)$ for $b \in B_0\setminus H$.  For simplicity we denote $\tZ_{\beta_{\xi=0}}$ by $\tZ$, $\tZ_{\beta_0}$ by $\tilde{U}_1$, $\tZ_{\beta_{k+1}}$ by $\tilde{U}_2$, $\tZ_{\beta_{z=\infty}}$ by $\tilde{V}_1$, and $\tZ_{\beta_{w=\infty}}$ by $\tilde{V}_2$.  We have $\tZ \tZ_{\beta_{\xi=\infty}} = q^{\beta_{\xi=0}+\beta_{\xi=\infty}}$ is a constant (since $\beta_{\xi=0}+\beta_{\xi=\infty} \in H_2(\widehat{G_{k,l}}^*)$), and for simplicity we set the constant to be $1$.  Thus $\tZ_{\beta_{\xi=\infty}} = \tZ^{-1}$.

Let $Z_{\cD}$ be the generating function of open Gromov--Witten invariants corresponding to a boundary divisor $\cD$ (Equation \eqref{eq:gen}).  
By Theorem \ref{thm:G-oGW_wc} there is no wall-crossing for the disc classes $\beta_{\xi=0}$ and $\beta_{\xi=\infty}$, and they are the only disc classes of Maslov index $2$ and intersecting $\cD_{\xi=0}$ and $\cD_{\xi=\infty}$ exactly once respectively.  Hence $Z_{\cD_{\xi=0}} = \tZ$ and $Z_{\cD_{\xi=\infty}} = \tZ^{-1}$.  For simplicity we denote $Z_{\cD_{\xi=0}}$ by $Z$, and hence $Z = \tZ$ (meaning that the coordinate $Z$ does not need quantum corrections).

Theorem \ref{thm:G-oGW_toric} gives the open Gromov--Witten invariants of moment-map tori, which then gives the open Gromov--Witten invariants of fibers of our Lagrangian fibration by Theorem \ref{thm:G-oGW_wc}.  Then by some nice combinatorics which also appears in \cite[Proof of Corollary 4.3]{LLW}, the generating functions factorizes as follows:
\begin{enumerate}
\item $Z_{\cD_{z=1}} = \tilde{U}_1$ and $Z_{\cD_{w=1}} = \tilde{U}_2$ over $C_{--}$, 
\item $Z_{\cD_{z=1}} = \tilde{U}_1(1+Z)(1+q_1Z)\dots(1+q_1\dots q_{k-1}Z)$ and $Z_{\cD_{w=1}} = \tilde{U}_2$ over $C_{+-}$, 
\item $Z_{\cD_{z=1}} = \tilde{U}_1$ and $Z_{\cD_{w=1}} = \tilde{U}_2(1+cZ)(1+q'_1cZ)\dots(1+q_1'\dots q_{l-1}'cZ)$ over $C_{-+}$, 
\item $Z_{\cD_{z=1}} = \tilde{U}_1(1+Z)(1+q_1Z)\dots(1+q_1\dots q_{k-1}Z)$ and $Z_{\cD_{w=1}} = \tilde{U}_2(1+cZ)(1+q'_1cZ)\dots(1+q_1'\dots q_{l-1}'cZ)$ over $C_{++}$. 
\item $Z_{\cD_{z=\infty}} = \tilde{U}_1^{-1}$ over $C_{++} \cup C_{+-}$, 
\item $Z_{\cD_{z=\infty}} = \tilde{U}_1^{-1} (1+Z)(1+q_1Z)\dots(1+q_1\dots q_{k-1}Z)$ over $C_{-+} \cup C_{--}$, 
\item $Z_{\cD_{w=\infty}} = \tilde{U}_2^{-1}$ over $C_{++} \cup C_{-+}$, 
\item $Z_{\cD_{w=\infty}} = \tilde{U}_2^{-1} (1+cZ)(1+q'_1cZ)\dots(1+q_1'\dots q_{l-1}'cZ)$ over $C_{+-} \cup C_{--}$.
\end{enumerate}
Therefore we conclude that the ring generated by the functions $Z=Z_{\cD_{\xi=0}}$, $Z_{\cD_{\xi=\infty}}=Z^{-1}$, $U_1:=Z_{\cD_{z=1}}$, $V_1:=Z_{\cD_{z=\infty}}$, $U_2:=Z_{\cD_{w=1}}$, $V_2:=Z_{\cD_{w=\infty}}$ is 
the polynomial ring $\C[U_1,U_2,V_1,V_2,Z,Z^{-1}]$ mod out by the relations 
\begin{align}
U_1V_1&=(1+Z)(1+q_1Z)\dots(1+q_1\dots q_{k-1}Z), \notag \\
U_2V_2&=(1+cZ)(1+q'_1cZ)\dots(1+q_1'\dots q_{l-1}'cZ).\notag 
\end{align}
This completes the proof of the theorem.  
\end{proof}

\begin{Rem} \label{DifferentCrepantResolutions}
If a different crepant resolution is taken, the mirror takes the same form as above 
while the coefficients of the polynomials on the right hand side are different functions of the K\"ahler parameters $q_i$ and $q_i'$.  
They correspond to different choices of limit points (and hence different flat coordinates) over the complex moduli.
\end{Rem}


\subsection{SYZ from $\widetilde{G_{k,l}}$ to $\widehat{O_{k,l}}$} \label{sec:smGtoO}
The deformed generalized conifold $\widetilde{G_{k,l}}$ is given by 
$$
\Big\{(x,y,z,w) \in \C^2 \times (\C\setminus \{1\})^2 \ \big| \  xy-\sum_{i=0}^k\sum_{j=0}^{l}a_{i,j}z^iw^j=0\Big\} 
$$
for generic coefficients $a_{i,j} \in \C$.  
It is a conic fibration over the second factor $(\C\setminus \{1\})^2$ with discriminant locus being the Riemann surface $\Sigma_{k,l} \subset (\C\setminus \{1\})^2$ 
defined by the equation $\sum_{i=0}^k\sum_{j=0}^{l}a_{i,j}z^iw^j=0$ 
which has genus $kl$ and $(k+l)$ punctures.  The SYZ construction for such a conic fibration follows from \cite{AAK}.  Here we just give a brief description.
We will use the standard symplectic form on $\C^2 \times (\C\setminus \{1\})^2$ restricted to the hypersurface $\widetilde{G_{k,l}}$. 
First, $\widetilde{G_{k,l}}$ is naturally compactified in $\PP^2\times(\PP^1)^2$ as a symplectic manifold, and we denote the compactification by $\widetilde{G_{k,l}}^*$. 
There is a natural Hamiltonian $T^1$-action on $\widetilde{G_{k,l}}^*$ given by, for $t \in T^1\subset \C$
$$
t\cdot (x,y,z,w):=(tx,t^{-1}y,z,w).
$$
By carefully analyzing the symplectic reduction of this $T^1$-action, 
a Lagrangian torus fibration $\pi:\widetilde{G_{k,l}}^* \to B:=[-\infty,\infty]^3$ was constructed in \cite[Section 4]{AAK}. 
Topologically the fibration is the homeomorphic to the naive one given by
$$
(x,y,z,w) \mapsto (b_1,b_2,b_3) = \Big(\log|z|, \log|w|,\frac{1}{2}(|x|^2-|y|^2)\Big).
$$
However since the symplectic form induced on the symplectic quotient is not the standard one on $\PP^2$, it has to be deformed to give a Lagrangian fibration. 

The discriminant locus of this fibration consists of the boundary of $B$ and $\overline{A_{k,l}} \times \{0\}$, 
where $\overline{A_{k,l}} \subset [-\infty,\infty]^2$ is the compactification of the amoeba $A_{k,l} \subset \R^2_{\ge0}$ of the Riemann surface $\Sigma_{k,l}$ (Figure \ref{fig:ConicAmoeba}), 
namely the image of $\Sigma_{k,l}$ under the map $(z,w) \mapsto (\log|z|,\log|w|)$. 
\begin{figure}[htbp]
 \begin{center} 
  \includegraphics[width=100mm]{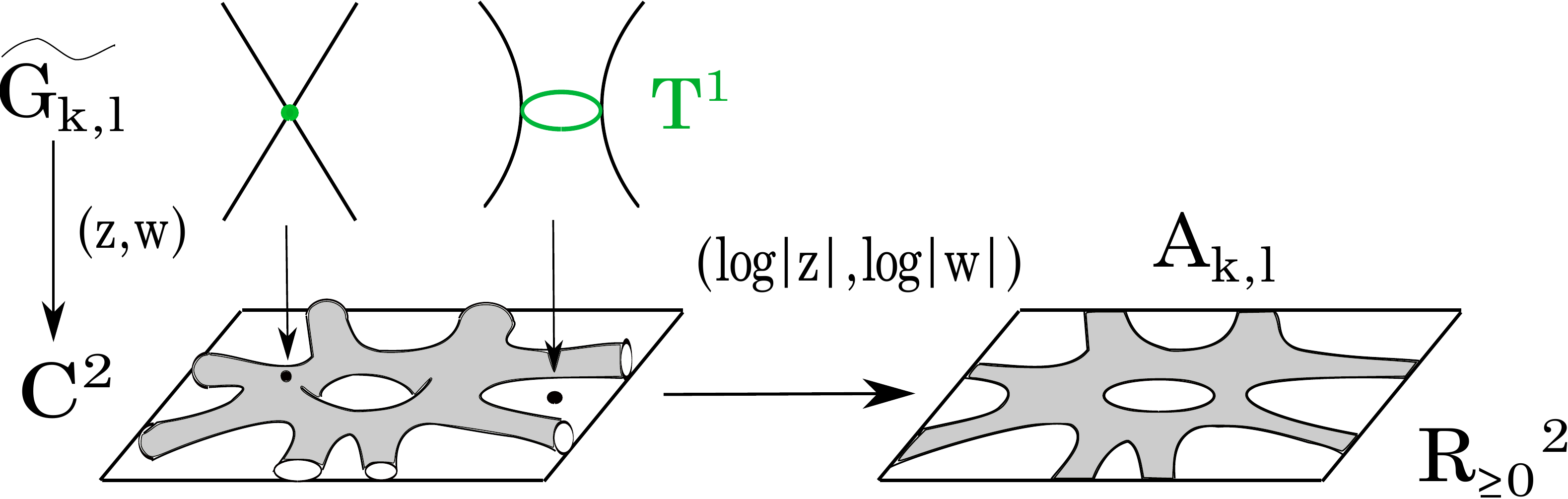}
 \end{center}
 \caption{Conic fibration and amoeba $A_{k,l}$}
\label{fig:ConicAmoeba}
\end{figure}
The wall for open Gromov--Witten invariants is given by $H=\overline{A_{k,l}}\times [-\infty,\infty]$ \cite[Proposition 5.1]{AAK}. 
The complement $B \setminus H$ consists of $(k+1)(l+1)$ chambers.  In this specific case, we have a nice degeneration as follows.  
Let us consider a special point on the the complex moduli space of $\widetilde{G_{k,l}}^*$ 
where the defining equation of the Riemann surface $\Sigma_{k,l}$ factorizes as
$$
\sum_{i=0}^k\sum_{j=0}^{l}a_{i,j}z^iw^j=f(z)g(w)
$$
for polynomials $f(z)$ and $g(w)$ of degree $k$ and $l$ respectively (and we assume that their roots are all distinct and non-zero). 
At this point, $\widetilde{G_{k,l}}^*$ acquires $kl$ conifold singularities.  The wall becomes the union of vertical hyperplanes $\{b_2=\log|r_i|\}_{i=1}^k \cup \{b_3=\log|s_j|\}_{j=1}^l \subset \R^3$, 
where $r_i$ and $s_j$ are the roots of $f(z)$ and $g(w)$ respectively (Figure \ref{fig:AmoebaConi}). 
\begin{figure}[htbp]
 \begin{center} 
  \includegraphics[width=70mm]{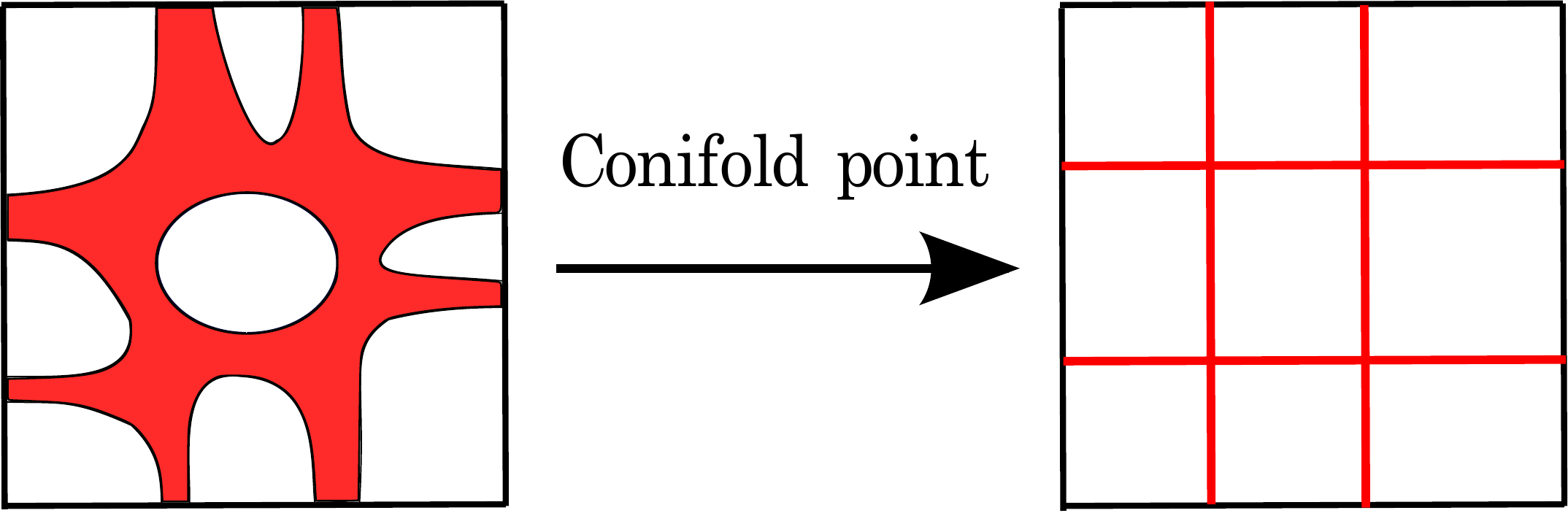}
 \end{center}
 \caption{Amoeba around conifold locus}
\label{fig:AmoebaConi}
\end{figure}
These hyperplanes divide the base into $(k+1)(l+1)$ chambers. 
We label the chambers by $C_{i,j}$ for $i = 0, \ldots, k$ and $j = 0, \ldots, l$ from left to right and from bottom to up.

\begin{Thm} \label{def G_kl to res O_kl}
The SYZ mirror of the deformed generalized conifold $\widetilde{G_{k,l}}$ is given by, for $U_i,V_i \in \C$ and $Z \in \C^\times$, 
$$
U_1V_1=(1+Z)^k, \ \ \ U_2V_2=(1+Z)^l, 
$$ 
which is the punctured orbifolded conifold $O_{k,l}$.
\end{Thm}
\begin{proof}
The wall-crossing of open Gromov--Witten invariants was deduced in \cite[Lemma 5.4]{AAK}, and we just sketch the result here. 
For $p=1,2$, let $U_p$ be the generating function of open Gromov--Witten invariants for disc classes intersecting the boundary divisor $\pi^{-1}(\{b_p=0\})$ once.  Denote the semi-flat coordinates corresponding to the basic disc classes emanated from $\pi^{-1}(\{b_p=0\})$ by $\tilde{Z}_i$, and denote by $Z$ the semi-flat coordinate corresponding to the $b_3$-direction (which admits no quantum corrections).  
Then $U_1$ restricted to the chamber $C_{i,j}$ equals to the polynomial $\tilde{Z}_1 (1+Z)^i$, and $U_2$ restricted to the chamber $C_{i,j}$ equals to $\tilde{Z}_2 (1+Z)^j$. 
By gluing the various chambers together using the above wall-crossing factor $1+Z$, we obtain the SYZ mirror as claimed. 
\end{proof}

\begin{Rem} \label{Remark on complex structure}
The equation in Theorem \ref{def G_kl to res O_kl} defines a singular variety $O_{k,l}$. 
This is a typical feature of our SYZ construction, which produces a complex variety out of a symplectic manifold: 
we may obtain a singular variety as the SYZ mirror, and we need to take a crepant resolution to get a smooth mirror. 
Since we concern about the complex geometry of this variety, $O_{k,l}$ is not distinguishable from its crepant resolution $\widehat{O_{k,l}}$. 
\end{Rem}


\subsection{SYZ from $\widehat{O_{k,l}}$ to $\widetilde{G_{k,l}}$} \label{sec:OtoG}
The partial compactification $\widehat{O_{k,l}^\sharp}$ of $\widehat{O_{k,l}}$ is a toric Calabi--Yau threefold 
and its SYZ mirror was constructed in \cite{CLL}.  In this section we quote the relevant results, omitting the details. 
First, a crepant resolution $\widehat{O_{k,l}}$ of $\widehat{O_{k,l}}$ corresponds to a maximal triangulation of $\Box_{k,l}$ (Figure \ref{fig:OrbifoldedConifold}(b)). 
We have the Lagrangian torus fibration $\pi:\widehat{O_{k,l}} \rightarrow B:=\R^2\times \R_{\ge0}$ constructed in \cite{Gol,Gro}, whose discriminant locus consists of two components. 
One is the boundary $\partial B$, and the other is topologically given by the dual graph of the maximal triangulation lying in the hyperplane $\{b_3=1\} \subset B$, 
where we denote the coordinates of $B$ by $b = (b_1,b_2,b_3)$. 

The wall is exactly the hyperplane $\{b_3=1\}$ containing one component of the discriminant locus.  It is associated with a wall-crossing factor, which is a polynomial whose coefficients encode the information coming from holomorphic discs with Maslov index $0$. The explicit formula for the coefficients were computed in \cite{CCLT}.  Applying these results, we obtain the following:

\begin{Thm}
The SYZ mirror of the resolved orbifolded conifold $\widehat{O_{k,l}}$ is a deformed generalized conifold $\widetilde{G_{k,l}}$ given as
$$
\Big\{ (U,V,Z,W) \in \C^2 \times (\C^\times)^2 \ \big| \ UV = \sum_{i=0}^k \sum_{j=0}^l  q^{C_{ij}} (1+\delta_{ij}(\check{q})) Z^i W^j \Big\}.
$$
\end{Thm}
The notations are explained as follows.  Let $\beta_{ij}$ be the basic disc class corresponding to the toric divisor $D_{ij} \subset \widehat{O_{k,l}}$. 
Then $C_{ij}$ denotes the curve class $\beta_{ij} - i (\beta_{10} - \beta_{00}) - j (\beta_{01} - \beta_{00}) - \beta_{00}$. 
The coefficients $1+\delta_{ij}(\check{q})$ is given by $\exp(g_{ij}(\check{q}))$
where
$$
g_{ij}(\check{q}):=\sum_{d}\frac{(-1)^{(D_{ij}\cdot d)}(-(D_{ij}\cdot d)-1)!}{\prod_{(a,b) \neq (i,j)} (D_{ab}\cdot d)!}\check{q}^d,
$$
and the summation is over all effective curve classes $d\in H_2^\text{eff}(O_{k,l})$ satisfying $D_{ij}\cdot d<0 \text{ and } D_p\cdot d \geq 0 \text{ for all } p\neq (i,j)$. 
Lastly $q$ and $\check{q}$ are related by the mirror map:
$$ q^C = \check{q}^C \exp \Big(- \sum_{i,j} (D_{ij} \cdot C) g_{ij}(\check{q}) \Big).$$


It is worth noting that the above SYZ mirror manifold can be identified with the Hori--Iqbal--Vafa mirror manifold \cite{HIV}. 
The former has the advantage that it is intrinsically expressed in terms of flat coordinates and contains the information about certain open Gromov--Witten invariants. 


\subsection{SYZ from $\widetilde{O_{k,l}}$ to $\widehat{G_{k,l}}$} \label{smOtoG}
Recall that the fan polytope of the orbifolded conifold $O_{k,l}^\sharp$ is the cone over a rectangle $[0,k] \times [0,l]$. 
Smoothings of $\widetilde{O_{k,l}}$ correspond to the Minkowski decompositions of $[0,k] \times [0,l]$ into $k$ copies of $[0,1] \times \{0\}$ and $l$ copies of $\{0\} \times [0,1]$ \cite{Alt}. 
The SYZ mirrors for such smoothings were constructed in \cite{Lau}.  
Here we can write down the Lagrangian fibration more explicitly by realizing $\widetilde{O_{k,l}}$ as a double conic fibration.
Recall that the deformed orbifolded conifold $\widetilde{O_{k,l}}$ is defined as 
$$
\widetilde{O_{k,l}}=\{(u_1,u_2,v_1,v_2,z) \in \C^4 \times \C^\times \ | \  u_1v_1=f(z), \ u_2v_2=g(z)\}$$
where $f(z)$ and $g(z)$ are generic polynomials of degree $k$ and $l$ respectively.  
We assume that all roots $r_i$ and $s_j$ of $f(z)$ and $g(z)$ respectively are distinct and non-zero. 
Moreover, we can naturally compactly $\widetilde{O_{k,l}}$ in $(\PP^2)^2 \times \PP^1$ to obtain $\widetilde{O_{k,l}}^* \subset (\PP^2)^2 \times \PP^1$ (where $(u_1,u_2)$ and $(v_1,v_2)$ above become inhomogeneous coordinates of the two $\PP^2$ factors.).  
There is also a natural Hamiltonian $T^2$-action on $\widetilde{O_{k,l}}^*$ given by, for $(s,t) \in T^2 \subset \C^2$ 
$$
(s,t)\cdot (u_1,v_1,u_2,v_2,z):=(su_1,s^{-1}v_1,tu_2,t^{-1}v_2,z).
$$
On the other hand, $\widetilde{O_{k,l}}^*$ admits a double conic fibration $\pi_z:\widetilde{O_{k,l}}^*\rightarrow \PP^1$ by the projection to the $z$-coordinate (Figure \ref{fig:ConicBundle}).  
\begin{figure}[htbp]
 \begin{center} 
  \includegraphics[width=70mm]{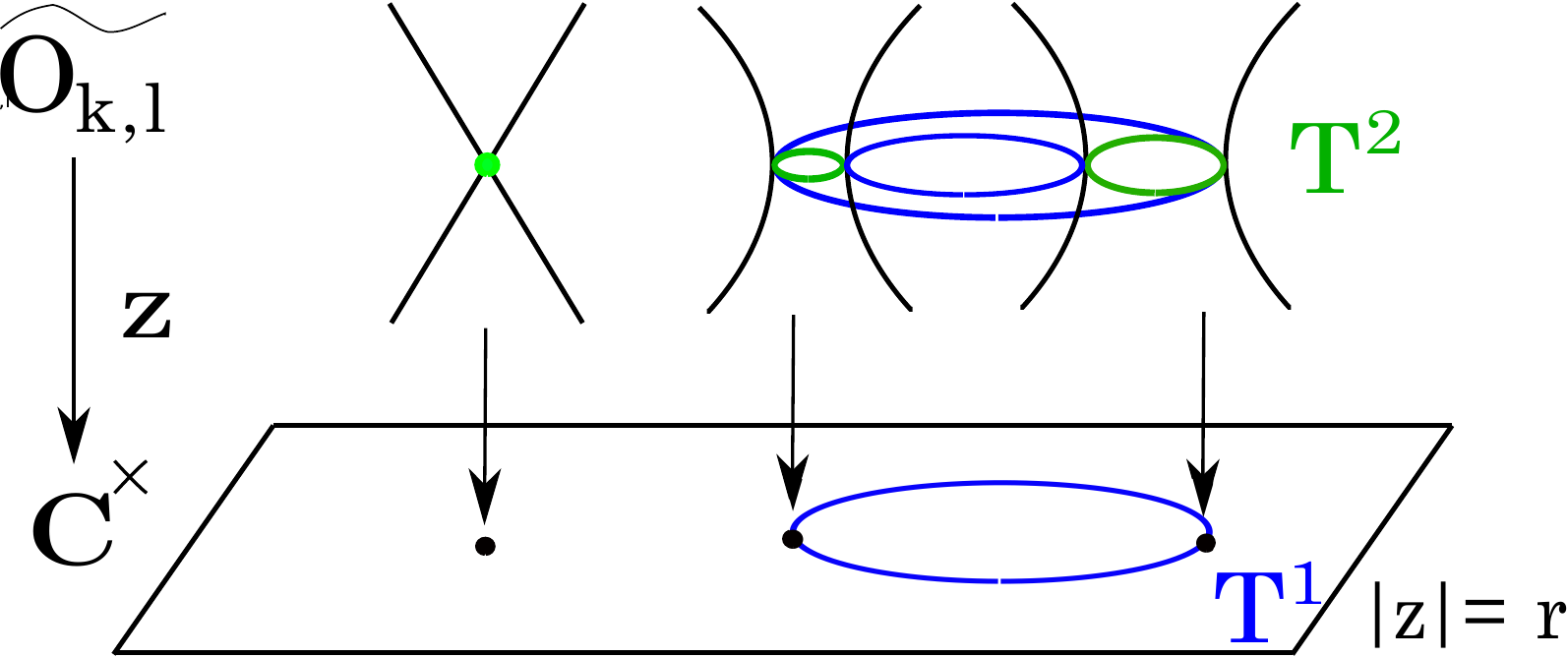}
 \end{center}
 \caption{Double conic fibration}
\label{fig:ConicBundle}
\end{figure}
In this situation, the base $\PP^1$ of the fibration can be identified as the symplectic reduction of $\widetilde{O_{k,l}}^*$ by the Hamiltonian $T^2$-action.  
As is discussed in \cite{Gro}, the Lagrangian fibration $|z|: \PP^1 \to [0,\infty]$ gives rise to the Lagrangian torus fibration $\pi:\widetilde{O_{k,l}}^* \rightarrow B:=[-\infty,\infty]^2 \times [0,\infty]$ given by 
$$
\pi(u_1,v_1,u_2,v_2,z)=(\frac{1}{2}(|u_1|^2-|v_1|^2),\frac{1}{2}(|u_2|^2-|v_2|^2),|z|).
$$
The map to the first two coordinates is the moment map of the Hamiltonian $T^2$-action. 
We denote the coordinates of $B$ by $b = (b_1,b_2,b_3)$. 

\begin{Prop} \label{DiscDefOrbConi}
\begin{enumerate}
\item The discriminant locus of the fibration $\pi$ is given by the disjoint union 
$\partial B \cup \left(\bigcup_{i=1}^k \{b_1=0,b_3=|r_i|\}\right) \cup \left(\bigcup_{j=1}^l \{b_2=0,b_3=|s_j|\}\right) \subset B$. 
\item The fibration $\pi$ is special with respect to the nowhere-vanishing meromorphic volume form $\Omega:=du_1 \wedge du_2 \wedge d\log z$ on $\widetilde{O_{k,l}}^*$. 
\end{enumerate}
\end{Prop}
\begin{proof}
The fibration has tori $T^3$ as generic fibers.  Over $\partial B$ where $z=0$, the fibers degenerate to $T^2$.  Thus $\partial B$ is a component of the discriminant locus.  Away from $z=0$, the map $z \to |z|$ is a submersion.  Hence the discriminant locus of the fibration $\pi$ comes from that of the moment map of the Hamiltonian $T^2$-action.  This action has non-trivial stabilizers at $u_1=v_1=0$ or $u_2=v_2=0$, which implies $f(z)=0$ or $g(z)=0$ respectively.  
Their images under $\pi$ are $\{b_1=0,b_3=|r_i|\}$ or $\{b_2=0,b_3=|s_j|\}$ respectively.
\end{proof}

\begin{Prop} \label{WallDefOrbConi}
A regular fiber of the fibration $\pi$ bounds a holomorphic disc of Maslov index $0$ only when $b_3 = |r_i|$ or $b_3 = |s_j|$. 
Thus the wall $H$ of the fibration $\pi$ is $H= \left(\bigcup_{i=1}^k \{b_3 =|r_i|\}\right) \cup \left(\bigcup_{j=1}^l \{b_3 =|s_j|\}\right) \subset B$. 
\end{Prop}
\begin{proof}
A singular fiber of the double conic fibration $\pi_z:\widetilde{O_{k,l}}^*\rightarrow \PP^1$ bounds a holomorphic disc, 
which has Maslov index $0$ by Lemma \ref{MaslovIndex}, 
and this happened only when $b_3 = |r_i|$ or $b_3 = |s_j|$. 
\end{proof}

\begin{figure}[htbp]
 \begin{center} 
  \includegraphics[width=35mm]{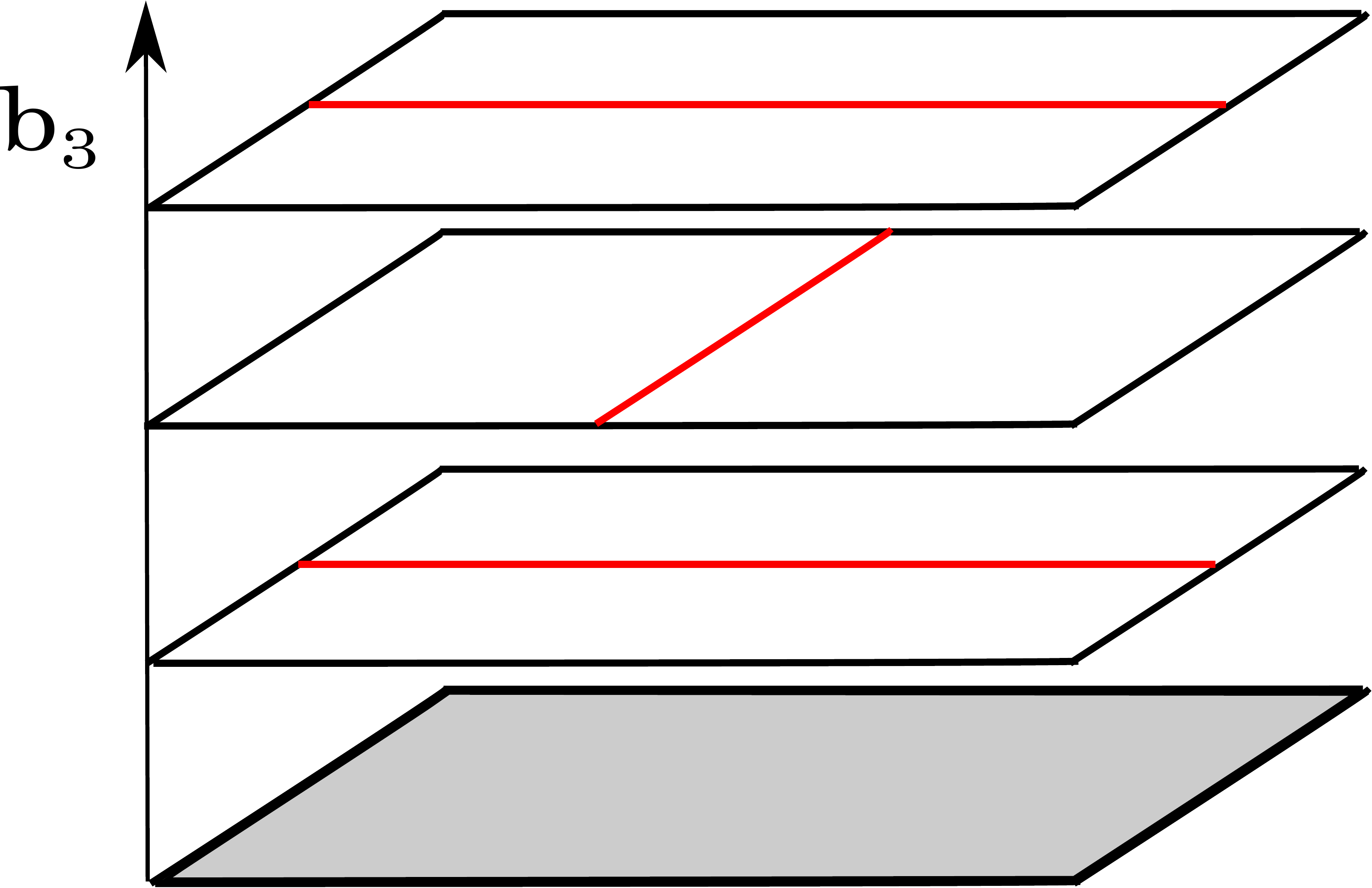}
 \end{center}
 \caption{Walls $(k,l)=(2,1)$}
\label{fig:Wall}
\end{figure}

The wall components $\{b_3 =|r_i|\}$ and $\{b_3 =|s_j|\}$ correspond to the pieces $[0,1]\times\{0\}$ and $\{0\}\times[0,1]$ of the Minkowski decomposition respectively. 
Wall-crosing of open Gromov--Witten invariants in this case has essentially been studied in \cite{Lau} in details, and we will not repeat the details here.  The key result is that each wall component contributes a linear factor: each component $\{b_3 =|r_i|\}$ contributes $1+X$, and each component $\{b_3 =|s_j|\}$ contributes $1+Y$. 
The SYZ mirror is essentially the product of all these factors, namely, we obtain the following. 

\begin{Thm} \label{def O_kl to res G_kl}
The SYZ mirror of $\widetilde{O_{k,l}}$ is given by 
$$
UV=(1+X)^k(1+Y)^l
$$
for $U, V \in \C$ and $X, Y \in \C^\times$, which is the punctured generalized conifold $G_{k,l}$.
\end{Thm}



An almost the same remark as Remark \ref{Remark on complex structure} applies to Theorem \ref{def O_kl to res G_kl} and thus we confirm the SYZ construction in this case. 



\section{Global geometric transitions: some speculations} \label{Discussion} 
We are now in position to turn to the global case. 
Let $\widehat{X}$ and $\widetilde{X}$ be compact Calabi--Yau threefolds. 
We call a geometric transition $\widehat{X} \dashrightarrow  X \rightsquigarrow \widetilde{X}$ a generalized conifold transition if 
$X$ has only generalized conifolds and orbifolded conifolds. 
The birational contraction appearing in the geometric transition of compact Calabi--Yau threefolds 
can be factorized into a sequence of primitive contractions of type I, type II and type III \cite{Ros}. 
In general, type I and type III appear in the geometric transition of $G_{k,l}$ and all types appear in the geometric transition of $O_{k,l}$. 

Motivated by the local case, we are tempted to propose that {\it generalized conifold transitions are reversed under mirror symmetry}. 
However, this naive conjecture does not hold because some global conifold transitions are mirror to hyperconifold transitions, which are not generalized conifold transitions \cite{Dav}.  We expect that a generalized conifold transition is mirror to a reversed generalized conifold transition if the Calabi--Yau threefold has a Lagrangian torus fibration and the transition is locally modeled by those given in Section \ref{SYZ}.



\begin{Ex}[Schoen's CY threefold and its mirror \cite{Sch, Lau2}]
Using the methods of Castano-Bernard and Matessi \cite{CM1,CM} in the Gross--Siebert program, generalized conifold transitions and their mirrors are studied for the Schoen's Calabi--Yau threefold in \cite{Lau2}.  
The threefold is a resolution of the fiber product of two rational elliptic surfaces over the base $\bP^1$ \cite{Sch}.  
It gives a global manifestation that orbifolded conifolds and generalized conifolds are mirror to each other.

For each pair of reflexive polygons $(P_1,P_2)$ (\emph{where $P_1$ and $P_2$ are not necessarily dual to each other}), we have an orbifolded conifold degeneration $O^{(P_1,P_2)}$ of a Schoen's Calabi--Yau threefold, and a generalized conifold degeneration $G^{(P_1,P_2)}$ of its mirror (in the sense of Legendre transform in Gross--Siebert program \cite{GS}).  A resolution of $O^{(P_1,P_2)}$ is mirror to a smoothing of $G^{(\check{P}_1,\check{P}_2)}$, and vice versa.  $\check{P}$ denotes the dual polygon of $P$ (see Figure \ref{fig:SchoenCY} for an example). 
We refer the reader to the paper \cite{Lau2} for more details. 

\end{Ex}

\begin{figure}[htb!]
	\centering
	\begin{subfigure}[b]{0.45\textwidth}
		\centering
		\includegraphics[scale=0.4]{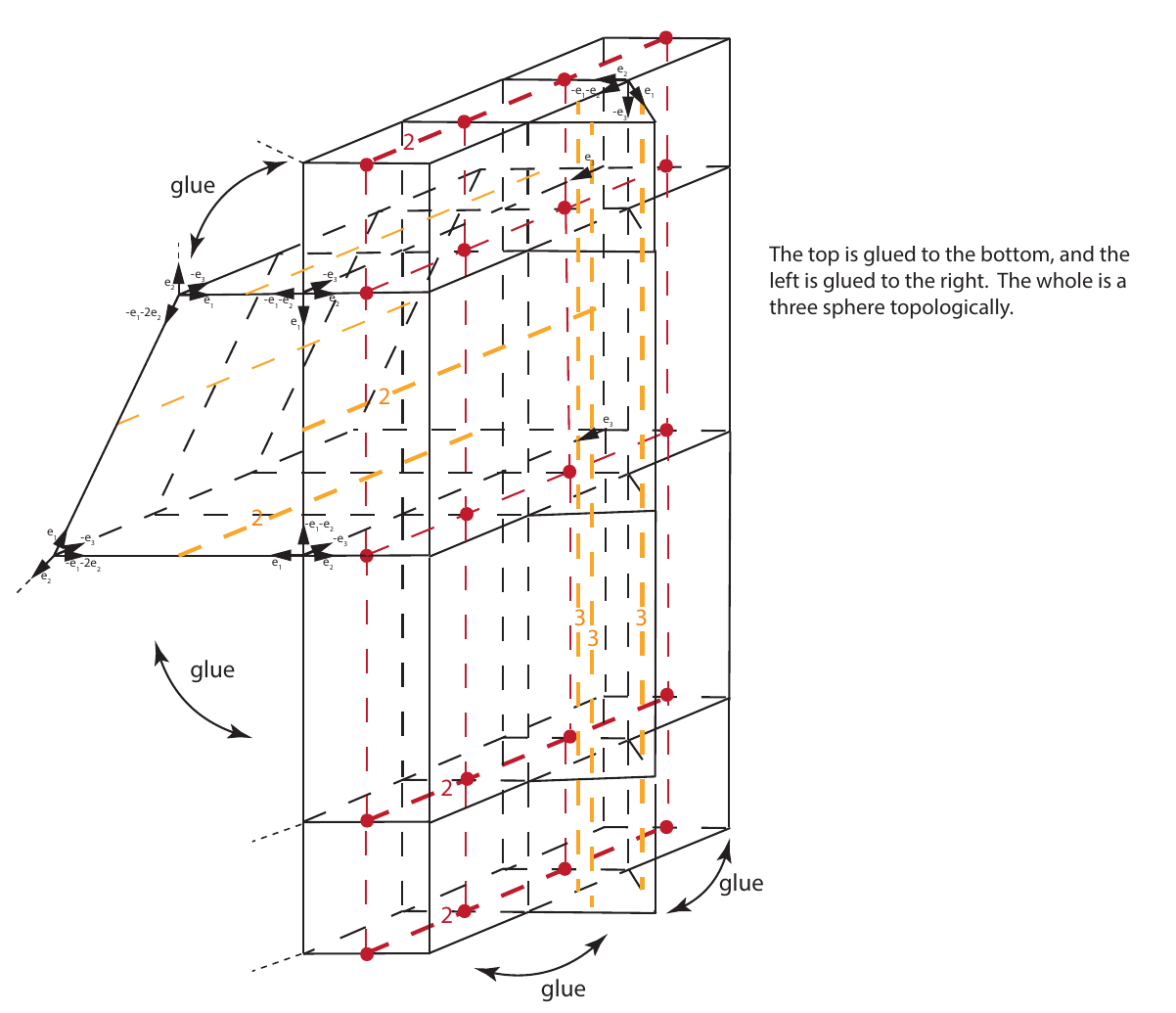}
		\caption{}
	\end{subfigure}
	\hspace{10pt}
	\begin{subfigure}[b]{0.45\textwidth}
		\centering
		\includegraphics[scale=0.4]{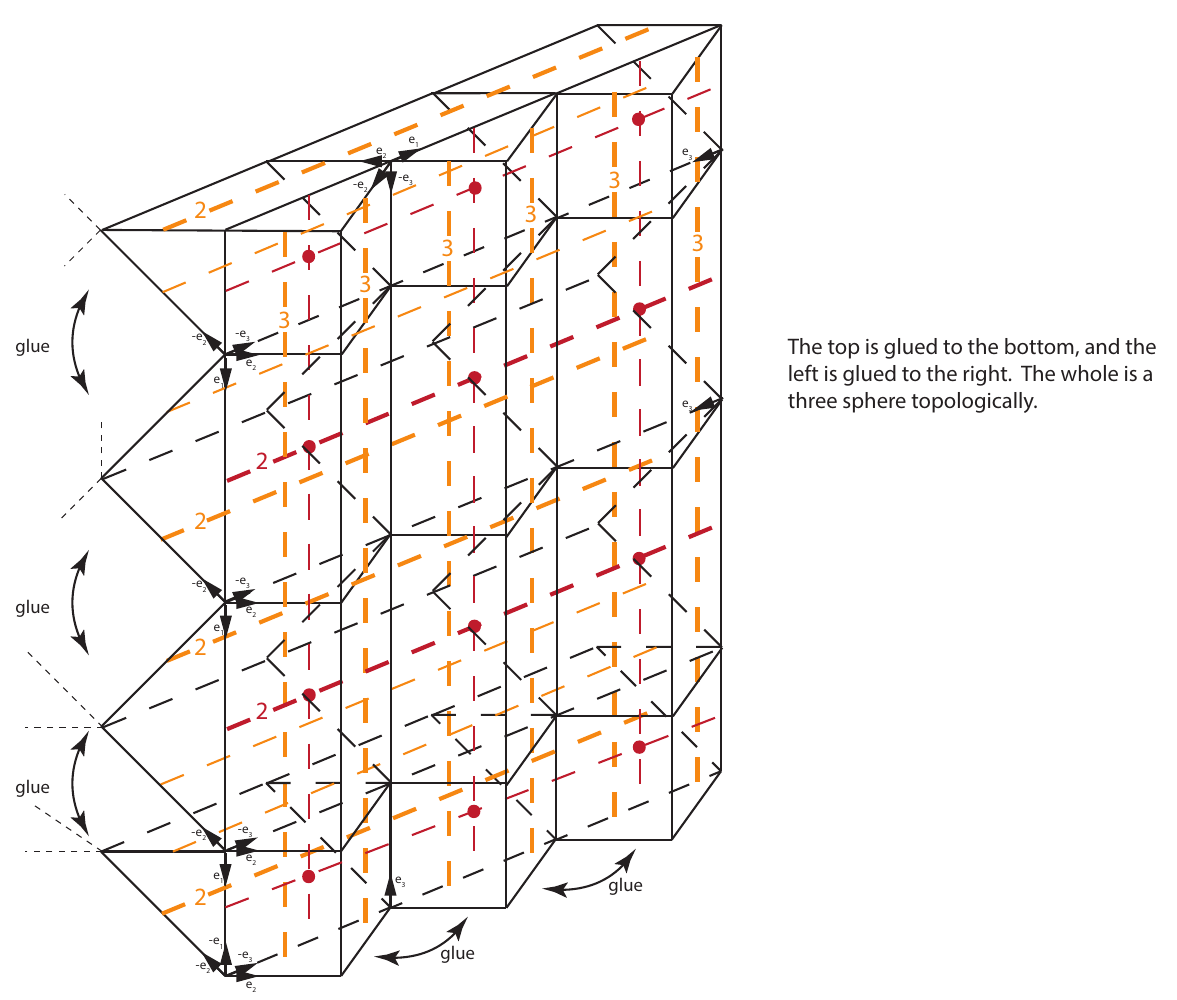}
		\caption{}
	\end{subfigure}
	\caption{Base of Lagrangian fibrations on degenerations of a Schoen's Calabi--Yau threefold.  Topologically they are $\bS^3$ and the figures show polyhedral decompositions which are useful to describe the affine structures.  The left shows an orbifolded conifold degeneration of Schoen's Calabi--Yau threefold.  The right shows its mirror.  Each thick dot represents an orbifolded conifold singularity on the left, and a generalized conifold singularity on the right.  There are 24 orbifold singularities counted with multiplicities.}
	\label{fig:SchoenCY}
\end{figure}

\begin{Ex}[Quintic threefold and its mirror]
We set $(k,l)=(4,1)$ or $(3,2)$ in the following. 
Let $X \subset \PP^4$ be the singular quintic threefold defined by
$$
x_0f(x_0,\dots,x_4)+x_1^kx_2^l=0,
$$
where $f(x)$ is a generic homogeneous polynomial of degree $4$.  
The singular locus of $X$ consists of 2 curves 
$$
\{x_0=x_1=f(x)=0\} \cup \{x_0=x_2=f(x)=0\} \subset X
$$
of genus $3$ intersecting at $4$ points. 
The quintic threefold $X$ has $G_{k,l}$ around the each intersection point. 
Successively blowing up $X$ along the two curves followed by the blow-up along the divisor $\{x_0=x_1=0\}$, we obtain a projective crepant resolution $\widehat{X}$ of $X$. 
Thus a quintic threefold admits a generalized conifold transition. 

On the other hand, the mirror quintic of a quintic threefold is defined as a crepant resolution of the orbifold 
$$
Y_{\phi}:=\Big\{\sum_{i=0}^4x_i^5+\phi\prod_{i=0}^4x_i=0\Big\}/G, \ \ \ G:=\Big\{(a_i)\in (\Z_5)^5 \ \big| \ \sum_{i=0}^4a_i=0 \Big\}/\Z_5
$$
for $\phi \in \C$. 
The orbifold $Y_{\phi}$ has $A_4$-singularities along 10 curves $C_{ij}=\{x_i=x_j=0\}/G \cong \PP^1, \ \ \ (0 \le i<j\le 4)$. 
We can partially resolve $Y_{\phi}$ to obtain $Y$ 
whose singular locus consists of $A_k$-singularities along $C_{01}$ and $A_l$-singularities along $C_{02}$ such that $Y$ has $O_{k,l}$ around $C_{01}\cap C_{02}$ (Figure \ref{fig:ToricSing}). 
\begin{figure}[htbp]
 \begin{center} 
  \includegraphics[width=50mm]{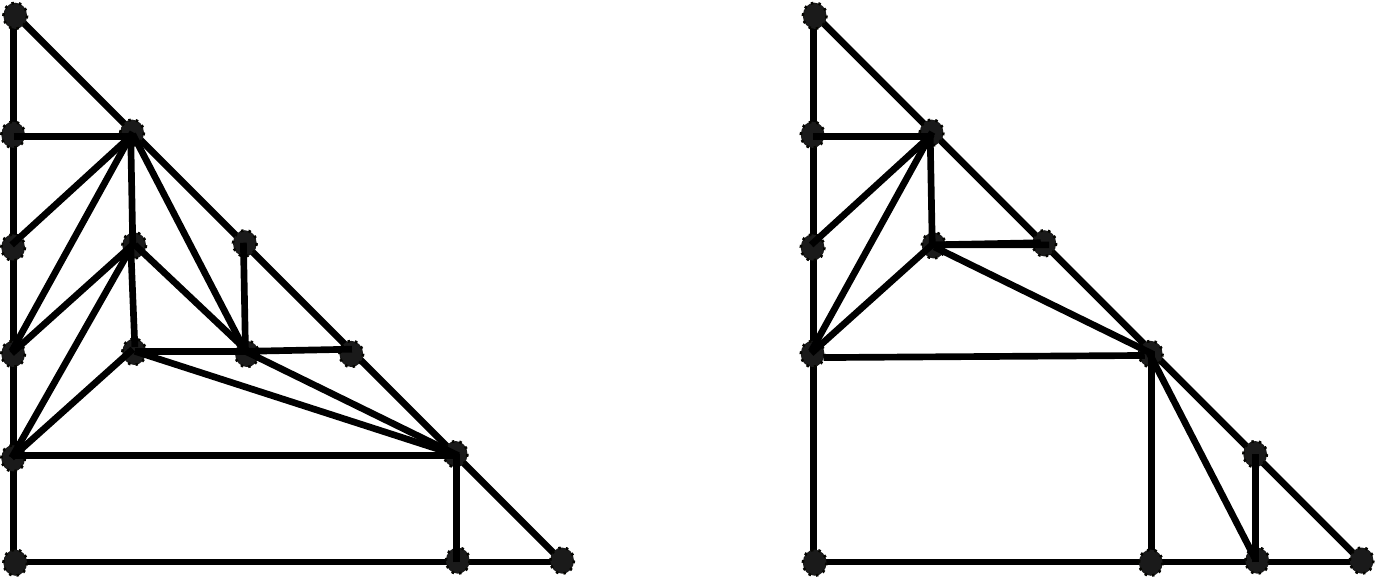}
 \end{center}
 \caption{$2$-dimensional faces of the polytope for $O_{4,1}$ and $O_{3,2}$} 
\label{fig:ToricSing}
\end{figure}
It is interesting to ask whether or not $Y$ admits any smoothing. 

Although $X$ and $Y$ lie in the boundaries of the complex moduli space of the quintic and the K\"ahler moduli space of the mirror quintic respectively, 
we do not know whether or not they correspond each other under the mirror correspondence. 
This may be seen by the monomial-divisor correspondence in toric geometry, 
but it is possible that the mirror of $X$ is a non-toric blow-down of the mirror quintic. 

It is straightforward to generalize this type of constructions to Calabi--Yau hypersurfaces in 4-dimensional weighted projective spaces.
\end{Ex}

We may also consider existence of generalized conifold transitions for compact Calabi--Yau geometries.
Let $\widehat{X}$ be a smooth threefold and $C_1,\dots, C_n$ be $(-1,-1)$-curves in $\widehat{X}$. 
Let $X$ be their contraction and $\widetilde{X}$ be a smoothing of $X$. 
Small resolutions and deformations always exist topologically, 
but there are obstructions if we wish to preserve either the complex or symplectic structure: 
\begin{Thm}[Friedman \cite{Fri}, Tian \cite{Tia}] \label{Friedman-Tian}
Assume that $X$ satisfies the $\partial \bar\partial$-lemma (for example K\"ahler). 
Then a smoothing $\widetilde{X}$ to exist if and only if there is a relation 
$
\sum_{i=1}^n \lambda_i [C_i]=0 \ (\lambda_i \ne 0 \ \forall i)
$
in $H_2(X,\Q)$. 
\end{Thm}
\begin{Thm}[Smith--Thomas--Yau \cite{STY}]
Let $\widetilde{Y}$ be a symplectic sixfold with embedded Lagrangian $S^3$s', say $L_1,\dots, L_n$. 
Then there is a relation
$
\sum_{i=1}^n \lambda_i [L_i]=0 \ (\lambda_i \ne 0 \ \forall i)
$
in $H_3(\widetilde{Y},\Q)$ if and only if there is a symplectic structure on one of $2^n$ choices of (reversed) conifold transitions of $\widetilde{Y}$ in the Lagrnagians $L_1,\dots,L_n$, 
such that the resulting exceptional $\PP^1$s' are symplectic. 
\end{Thm}
In our case, the contractions collapse 4-cycles as well as 2-cycles.  
On the other hand, smoothing $G_{k,l} \rightsquigarrow\widetilde{G_{k,l}}$ produces $(k+1)(l+1)-3$ vanishing $S^3$s' 
and smoothing $O_{k,l} \rightsquigarrow\widetilde{O_{k,l}}$ produces $k+l-2$ vanishing $S^1\times S^2$s' and one vanishing $S^3$. 
The generators of these cycles can be found by considering the standard double Riemann surface fibrations \cite{FHKV}(Figure \ref{fig:DoubleRiemannSurf}).  
\begin{figure}[htbp]
 \begin{center} 
  \includegraphics[width=85mm]{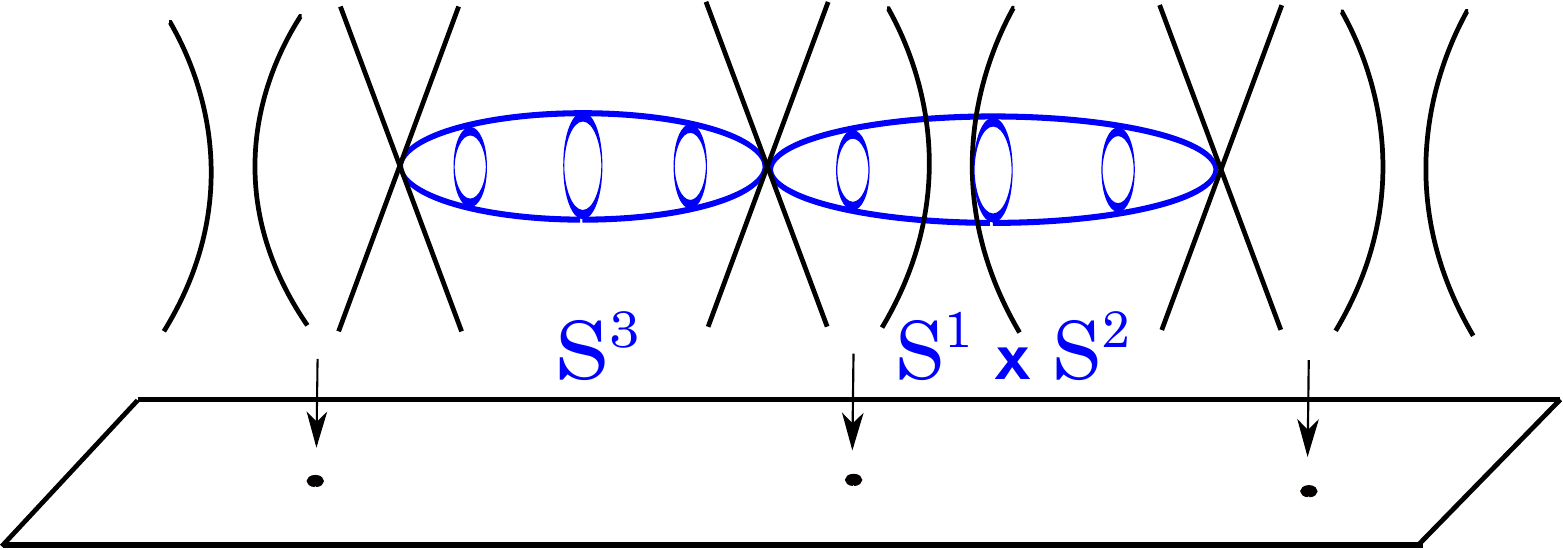}
 \end{center}
 \caption{A double Riemann surface fibration} 
\label{fig:DoubleRiemannSurf}
\end{figure}
It is interesting to investigate the obstructions to the deformations/resolutions of the generalized and orbifolded conifolds in terms of these cycles. 
We hope to come back to these questions in future work. 


\par\noindent{\scshape \small
Department of Mathematics, Kyoto University\\
Kitashirakawa-Oiwake, Sakyo, Kyoto, 606-8502, Japan}
\par\noindent{\ttfamily akanazawa@math.kyoto-u.ac.jp}\\

\par\noindent{\scshape \small
Department of Mathematics and Statistics, Boston University\\
111 Cummington Mall, Boston MA 02215 USA}
\par\noindent{\ttfamily lau@math.bu.edu}

\end{document}